\documentclass[letterpaper,11pt]{article}
\usepackage{color}

\usepackage[small]{titlesec}
\usepackage{theorem,amsmath,amssymb,amscd}
\usepackage[all]{xy}
\usepackage{booktabs}
\usepackage[hyperfootnotes=false, colorlinks, linkcolor={blue}, citecolor={magenta}, filecolor={blue}, urlcolor={blue}]{hyperref}
\theoremstyle{change}
\allowdisplaybreaks
\newcommand{\Z}{{\mathbb Z}}

\newcommand{\C}{{\mathbb C}}

\newcommand{\p}{\mathfrak p}
\newcommand{\OF}{{\mathfrak o}}
\newcommand{\GL}{{\rm GL}}

\newcommand{\GSp}{{\rm GSp}}

\newcommand{\St}{{\rm St}}
\newcommand{\triv}{{\mathbf1}}

\newcommand{\para}{K^{\rm para}}

\newcommand{\mat}[4]{{\setlength{\arraycolsep}{0.5mm}\left[
\begin{array}{cc}#1&#2\\#3&#4\end{array}\right]}}
\newcommand{\qed}{\hspace*{\fill}\rule{1ex}{1ex}}
\newcommand{\forget}[1]{}

\def\qdots{\mathinner{\mkern1mu\raise0pt\vbox{\kern7pt\hbox{.}}\mkern2mu
\raise3.4pt\hbox{.}\mkern2mu\raise7pt\hbox{.}\mkern1mu}}

\newenvironment{proof}{\vspace{1ex}\noindent{\it Proof.}\hspace{0.1em}}
	{\hfill\qed\vspace{2ex}}

\newtheorem{lemma}{Lemma.}[section]
\newtheorem{theorem}[lemma]{Theorem.}

\newtheorem{proposition}[lemma]{Proposition.}

\begin{document}
\thispagestyle{empty}
\begin{center}
 {\bf\Large Bessel models for $\GSp(4)$: Siegel vectors of square-free level}

 \vspace{3ex}
Ameya Pitale\footnote{{\tt apitale@math.ou.edu}\qquad $^2${\tt rschmidt@math.ou.edu}\\[0.5ex] \phantom{xxx}MSC: 11F46, 11F70\\ \phantom{xxx} The authors are supported by NSF grant DMS 1100541}, Ralf Schmidt$^2$

 \vspace{3ex}
 \begin{minipage}{80ex}
  \small{\sc Abstract}. We determine test vectors and explicit formulas for all Bessel models for those Iwahori-spherical representations of $\GSp_4$ over a $p$-adic field that have non-zero vectors fixed under the Siegel congruence subgroup.
 \end{minipage}
\end{center}

\section{Introduction}
For various classical groups, Bessel models of local or global representations have proven to be a useful substitute for the frequently missing Whittaker model. The uniqueness of Bessel models in the local non-archimedean case has now been established in a wide variety of cases; see \cite{AizGouRalSch2010}, \cite{GanGrossPrasad2012}. In this work we are only concerned with the group $\GSp_4$ over a $p$-adic field $F$, for which uniqueness of Bessel models was proven as early as 1973; see \cite{NovPia1973}.

Other than Whittaker models, which are essentially independent of any choices made, Bessel models depend on some arithmetic data. In the case of $\GSp_4$, part of this data is a choice of non-degenerate symmetric $2\times2$-matrix $S$ over the field $F$. The discriminant $\mathbf{d}$ of this matrix determines a quadratic extension $L$ of $F$; this extension may be isomorphic to $F\oplus F$, which will be referred to as the \emph{split case}. The second ingredient entering into the definition of a Bessel model is a character $\Lambda$ of the multiplicative group $L^\times$. 

Now let $(\pi,V)$ be an irreducible, admissible representation of $\GSp_4(F)$. Given $S$ and $\Lambda$, the representation $\pi$ may or may not have a Bessel model with respect to this data. In the case of $\GSp_4$, it is possible to precisely say which representations have which Bessel models; see \cite{PraTak2011}. In particular, every irreducible, admissible representation has a Bessel model for an appropriate choice of $S$ and $\Lambda$.

Given $\pi$, $S$ and $\Lambda$, it is one thing to know that a Bessel model exists, but it is another to identify a good \emph{test vector}. By definition, a test vector is a vector in the space of $\pi$ on which the relevant Bessel functional is non-zero. Equivalently, in the actual Bessel model consisting of functions $B$ on the group with the Bessel transformation property, $B$ is a test vector if and only if $B(1)$ is non-zero. In this paper, we will identify test vectors for a class of representations that is relevant for the theory of Siegel modular forms of degree $2$. In addition, we shall give explicit formulas for the corresponding Bessel functions. See \cite{Pitale2011}, \cite{Saha2009} for the Steinberg case and \cite{Sug1985} for the spherical case.

More precisely, the class of representations we consider are those that have a non-zero vector invariant under $P_1$, the Siegel ($\Gamma_0$-type) congruence subgroup of level $\p$. These representations appear as local components of global automorphic representations generated by Siegel modular forms of degree $2$ with respect to the congruence subgroup $\Gamma_0(N)$, where $N$ is a square-free positive integer. They fall naturally into thirteen classes, only four of which consist of generic representations (meaning representations that admit a Whittaker model); see our Table \ref{Iwahoritable} below.

Of the thirteen classes, six are actually spherical, meaning they have a non-zero vector invariant under the maximal compact subgroup $\GSp_4(\OF)$ (here, $\OF$ is the ring of integers of our local field $F$). Sugano \cite{Sug1985} has given test vectors and explicit formulas in the spherical cases. Our main focus, therefore, is on the seven classes consisting of non-spherical representations with non-zero $P_1$-fixed vectors. Of those, five classes have a one-dimensional space of $P_1$-fixed vectors, and two classes have a two-dimensional space of $P_1$-fixed vectors. The one-dimensional cases require a slightly different treatment from the two-dimensional cases. In all cases, our main tool will be two Hecke operators $T_{1,0}$ and $T_{0,1}$, coming from the double cosets
$$
 P_1\,{\rm diag}(\varpi,\varpi,1,1)P_1\qquad\text{and}\qquad P_1\,{\rm diag}(\varpi^2,\varpi,1,\varpi)P_1.
$$
Evidently, these operators act on the spaces of $P_1$-invariant vectors. In Sect.\ \ref{heckeeigenvaluessec} we give their eigenvalues for all of our seven classes of representations. The results are contained in Table \ref{Iheckeeigenvaluestable}. In the one-dimensional cases, trivially, the unique $P_1$-invariant vector is a common eigenvector for both $T_{1,0}$ and $T_{0,1}$. In the two-dimensional cases, it turns out there is a nice basis consisting of common eigenvectors for $T_{1,0}$ and $T_{0,1}$.

In Sect.\ \ref{T10sec} we will apply the two Hecke operators to $P_1$-invariant Bessel functions $B$ and evaluate at certain elements of $\GSp_4(F)$. Assuming that $B$ is an eigenfunction, this leads to several formulas relating the values of $B$ at various elements of the group; see Lemma \ref{T10actionlemma} as an example for this kind of result. The calculations are all based on a $\GL_2$ integration formula, which we establish in Lemma \ref{GL2integrationlemma1}.

In Sect.\ \ref{maintowersec} we use some of these formulas to establish a generating series for the values of $B$ at diagonal elements; see Proposition \ref{maintowerprop}. It turns out that there is one ``initial element''
$$
 h(0,m_0)={\rm diag}(\varpi^{2m_0},\varpi^{m_0},1,\varpi^{m_0}),
$$
where $m_0$ is the conductor of the Bessel character $\Lambda$; see (\ref{m0defeq}) for the precise definition. If $B(h(0,m_0))=0$, then $B$ is zero on \emph{all} diagonal elements.

A generating series like in Proposition \ref{maintowerprop} is precisely the kind of result that will be useful in global applications. However, it still has to be established that $B(h(0,m_0))\neq0$; this is the test vector problem, and it is not trivial since $B$ may vanish on all diagonal elements and yet be non-zero. It turns out that, in almost all cases, $B(h(0,m_0))\neq0$ as expected; see our main results Theorem \ref{onedimtheorem} and Theorem \ref{twodimtheorem}.

However, there is \emph{one} exceptional case, occurring only for split Bessel models of representations of type IIa, and then only for a certain unramified Bessel character $\Lambda$. In this very special case, our Bessel function $B$ is non-zero not at the identity, as expected, but at a certain other element; see Theorem \ref{onedimtheorem} i) for the precise statement.

To handle this exceptional case, and one other split case for VIa type representations, we require an additional tool besides our Hecke operators. This additional tool are zeta integrals, which are closely related to split Bessel models. We also require part of the theory of paramodular vectors from \cite{NF}. Roughly speaking, we take paramodular vectors and ``Siegelize'' them to obtain $P_1$-invariant vectors. Calculations with zeta integrals then establish the desired non-vanishing at specific elements. This is the topic of Sect.\ \ref{genericsplitsec}.

The final Sects.\ \ref{onedimsec} and \ref{twodimsec} contain our main results. Theorem \ref{onedimtheorem} exhibits test vectors in all one-dimensional cases, and Theorem \ref{twodimtheorem} exhibits test vectors in all two-dimensional cases. As mentioned above, explicit formulas for these vectors can be found in Proposition \ref{maintowerprop}. We mention that, evidently, explicit formulas imply uniqueness of Bessel models. Hence, as a by-product of our calculations, we reprove the uniqueness of these models for the representations under consideration.

Let us remark here that the proofs of several of the lemmas in this paper are very long but not conceptually very deep. Hence, we have omitted the proofs of these results. We direct the reader to \cite{Pitale-Schmidt-preprint-2012} for a longer version of this article which contains all the details of the proofs.
\section{Basic facts and definitions}\label{besselsubgroupsec}
Let $F$ be a non-archimedean local field of characteristic zero, $\OF$ its ring of integers, $\p$ the maximal ideal of $\OF$, and $\varpi$ a generator of $\p$. We fix a non-trivial character $\psi$ of $F$ such that $\psi$ is trivial on $\OF$ but non-trivial on $\p^{-1}$. We let $v$ be the normalized valuation map on $F$.

As in \cite{Fu1994} we fix three elements $\mathbf{a},\mathbf{b},\mathbf{c}$ in $F$ such that $\mathbf{d}=\mathbf{b}^2-4\mathbf{a}\mathbf{c}\neq0$. Let \begin{equation}\label{Sxidefeq}
 S=\mat{\mathbf{a}}{\frac{\mathbf{b}}2}{\frac{\mathbf{b}}2}{\mathbf{c}},\qquad
 \xi=\mat{\frac{\mathbf{b}}2}{\mathbf{c}}{-\mathbf{a}}{\frac{-\mathbf{b}}2}.
\end{equation}
Then $F(\xi)=F+F\xi$ is a two-dimensional $F$-algebra. If $\mathbf{d}$ is not a square in $F^\times$, then $F(\xi)$ is isomorphic to the field $L=F(\sqrt{\mathbf{d}})$ via the map $x+y\xi\mapsto x+y\frac{\sqrt{\mathbf{d}}}2$. If $\mathbf{d}$ is a square in $F^\times$, then $F(\xi)$ is isomorphic to $L=F\oplus F$ via $x+y\xi\mapsto(x+y\frac{\sqrt{\mathbf{d}}}2,x-y\frac{\sqrt{\mathbf{d}}}2)$. Let $z\mapsto\bar z$ be the obvious involution on $L$ whose fixed point set is $F$. The determinant map on $F(\xi)$ corresponds to the norm map on $L$, defined by $N(z)=z\bar z$. Let
\begin{equation}\label{TFdefeq}
 T(F)=\{g\in\GL_2(F):\:^tgSg=\det(g)S\}.
\end{equation}
One can check that $T(F)=F(\xi)^\times$, so that $T(F)\cong L^\times$ via the isomorphism $F(\xi)\cong L$. We define the Legendre symbol as
\begin{equation}\label{legendresymboldefeq}
 \Big(\frac L\p\Big)=\begin{cases}
                      -1&\text{if $L/F$ is an unramified field extension},\\
                      0&\text{if $L/F$ is a ramified field extension},\\
                      1&\text{if }L=F\oplus F.
                     \end{cases}
\end{equation}
These three cases are referred to as the \emph{inert case}, \emph{ramified case}, and \emph{split case}, respectively. If $L$ is a field, then let $\OF_L$ be its ring of integers and $\p_L$ be the maximal ideal of $\OF_L$. If $L = F \oplus F$, then let $\OF_L = \OF \oplus \OF$. Let $\varpi_L$ be a uniformizer in $\OF_L$ if $L$ is a field, and set $\varpi_L = (\varpi,1)$ if $L$ is not a field. In the field case let $v_L$ be the normalized valuation on $L$. 
Except in Sect.\ \ref{genericsplitsec}, where we consider certain split cases, we will make the following \emph{standard assumptions},
\begin{equation}\label{standardassumptions}
 \begin{minipage}{90ex}
  \begin{itemize}
   \item $\mathbf{a},\mathbf{b}\in\OF$ and $\mathbf{c}\in\OF^\times$.
   \item If $\mathbf{d}\notin F^{\times2}$, then $\mathbf{d}$ is a generator of the discriminant of $L/F$.
   \item If $\mathbf{d}\in F^{\times2}$, then $\mathbf{d}\in\OF^\times$.
  \end{itemize}
 \end{minipage}
\end{equation}
Under these assumptions, the group $T(\OF):=T(F)\cap\GL_2(\OF)$ is isomorphic to $\OF_L^\times$ via the isomorphism $T(F)\cong L^\times$. 

Under our assumptions (\ref{standardassumptions}) it makes sense to consider the quadratic equation $\mathbf{c}u^2+\mathbf{b}u+\mathbf{a}=0$ over the residue class field $\OF/\p$. The number of solutions of this equation is $\big(\frac L\p\big)+1$. In the ramified case we will fix an element $u_0\in\OF$ and in the split case we will fix two mod $\p$ inequivalent elements $u_1,u_2\in\OF$ such that
\begin{equation}\label{u1u2defeq}
 \mathbf{c}u_i^2+\mathbf{b}u_i+\mathbf{a}\in\p,\qquad i=0,1,2.
\end{equation}
%
%
\subsubsection*{Groups}
We define the group $\GSp_4$, considered as an algebraic $F$-group, using the symplectic form
$
 J=\mat{}{1_2}{-1_2}{}.
$
Hence, $\GSp_4(F)=\{g\in\GL_4(F):\:^tgJg=\mu(g)J\}$, where the scalar $\mu(g)\in F^\times$ is called the \emph{multiplier} of $g$. Let $Z$ be the center of $\GSp_4$. Let $B$ denote the Borel subgroup, $P$ the Siegel parabolic subgroup, and $Q$ the Klingen parabolic subgroup of $\GSp_4$.
Let $K=\GSp_4(\OF)$ be the standard maximal compact subgroup of $\GSp_4(F)$. The parahoric subgroups corresponding to $B$, $P$ and $Q$ are the \emph{Iwahori subgroup} $I$, the \emph{Siegel congruence subgroup} $P_1$, and the \emph{Klingen congruence subgroup} $P_2$, given by
\begin{equation}\label{parahoricsubgroupsdefeq}
 I=K\cap\begin{bmatrix}\OF&\p&\OF&\OF\\\OF&\OF&\OF&\OF\\\p&\p&\OF&\OF\\\p&\p&\p&\OF\end{bmatrix},\quad
 P_1=K\cap\begin{bmatrix}\OF&\OF&\OF&\OF\\\OF&\OF&\OF&\OF\\\p&\p&\OF&\OF\\\p&\p&\OF&\OF\end{bmatrix},\quad
 P_2=K\cap\begin{bmatrix}\OF&\p&\OF&\OF\\\OF&\OF&\OF&\OF\\\OF&\p&\OF&\OF\\\p&\p&\p&\OF\end{bmatrix}.
\end{equation}
We will also have occasion to consider, for a non-negative integer $n$, the \emph{paramodular group} of level $\p^n$, defined as
\begin{equation}\label{paradefeq}
 \para(\p^n)=\{g\in\GSp_4(F):\:g\in\begin{bmatrix}\OF&\p^n&\OF&\OF\\\OF&\OF&\OF&\p^{-n}\\\OF&\p^n&\OF&\OF\\\p^n&\p^n&\p^n&\OF\end{bmatrix},\:\det(g)\in\OF^\times\}.
\end{equation}
The eight-element Weyl group $W$ of $\GSp_4$, defined in the usual way as the normalizer modulo the centralizer of the subgroup of diagonal matrices, is generated by the images of
\begin{equation}\label{s1s2defeq}
 s_1=\begin{bmatrix}&1\\1\\&&&1\\&&1\end{bmatrix}\qquad\text{and}\qquad
 s_2=\begin{bmatrix}&&1\\&1\\-1\\&&&1\end{bmatrix}.
\end{equation}
%
\subsubsection*{Bessel models}
Let $\mathbf{a},\mathbf{b},\mathbf{c}$, the matrix $S$ and the torus $T(F)$ be as above. We consider $T(F)$ a subgroup of $\GSp_4(F)$ via
\begin{equation}\label{TFembeddingeq}
 T(F)\ni g\longmapsto\mat{g}{}{}{\det(g)\,^tg^{-1}}\in\GSp_4(F).
\end{equation}
Let $U(F)$ be the unipotent radical of the Siegel parabolic subgroup $P$
and $R(F)=T(F)U(F)$. We call $R(F)$ the \emph{Bessel subgroup} of $\GSp_4(F)$ (with respect to the given data $\mathbf{a},\mathbf{b},\mathbf{c}$). Let $\theta:\:U(F)\rightarrow\C^\times$ be the character given by
\begin{equation}\label{thetadefeq}
 \theta(\mat{1}{X}{}{1})=\psi({\rm tr}(SX)),
\end{equation}
where $\psi$ is our fixed character of $F$ of conductor $\OF$. 
We have $\theta(t^{-1}ut)=\theta(u)$ for all $u\in U(F)$ and $t\in T(F)$. Hence,
if $\Lambda$ is any character of $T(F)$, then the map
$tu\mapsto\Lambda(t)\theta(u)$ defines a character of $R(F)$. We denote
this character by $\Lambda\otimes\theta$. Let $\mathcal{S}(\Lambda,\theta)$ be the space of all locally constant functions $B:\:\GSp_4(F)\rightarrow\C$ with the \emph{Bessel transformation property}
\begin{equation}\label{Besseltransformationpropertyeq}
 B(rg)=(\Lambda\otimes\theta)(r)B(g)\qquad\text{for all $r\in R(F)$ and $g\in \GSp_4(F)$}.
\end{equation}
Our main object of investigation is the subspace $\mathcal{S}(\Lambda,\theta,P_1)$ consisting of functions that are right invariant under $P_1$. The group $\GSp_4(F)$ acts on $\mathcal{S}(\Lambda,\theta)$ by right translation. If an irreducible, admissible representation $(\pi,V)$ of $\GSp_4(F)$ is isomorphic to a subrepresentation of $\mathcal{S}(\Lambda,\theta)$, then this realization of $\pi$ is called a \emph{$(\Lambda,\theta)$-Bessel model}. It is known by \cite{NovPia1973}, \cite{PraTak2011} that such a model, if it exists, is unique; we denote it by $\mathcal{B}_{\Lambda,\theta}(\pi)$. Since the Bessel subgroup contains the center, an obvious necessary condition for existence is
$ \Lambda(z)=\omega_\pi(z)$ for all $z\in F^\times$,
where $\omega_\pi$ is the central character of $\pi$.
\subsubsection*{Change of models}
Of course, Bessel models can be defined with respect to any non-degenerate symmetric matrix $S$, not necessarily subject to the conditions (\ref{standardassumptions}) we imposed on $\mathbf{a},\mathbf{b},\mathbf{c}$. Since our calculations and explicit formulas will assume these conditions, we shall briefly describe how to switch to more general Bessel models. Hence, let $\lambda$ be in $F^\times$ and $A$ be in $\GL_2(F)$, and let $S'=\lambda\,^t\!ASA$. Replacing $S$ by $S'$ in the definitions (\ref{TFdefeq}) and (\ref{thetadefeq}), we obtain the group $T'(F)$  and the character $\theta'$ of $U(F)$. There is an isomorphism $T'(F)\rightarrow T(F)$ given by $t\mapsto AtA^{-1}$. Let $\Lambda'$ be the character of $T'(F)$ given by $\Lambda'(t)=\Lambda(AtA^{-1})$. For $B\in\mathcal{B}_{\Lambda,\theta}(\pi)$, let
$ B'(g)=B(\mat{A}{}{}{\lambda^{-1}\,^t\!A^{-1}}g)$, for $g\in \GSp_4(F)$.
It is easily verified that $B'$ has the $(\Lambda',\theta')$-Bessel transformation property, and that the map $B\mapsto B'$ provides an isomorphism $\mathcal{B}_{\Lambda,\theta}(\pi)\cong\mathcal{B}_{\Lambda',\theta'}(\pi)$.

Let $S=\mat{\mathbf{a}}{\frac{\mathbf{b}}2}{\frac{\mathbf{b}}2}{\mathbf{c}}$ be the usual matrix such that $\mathbf{d}=\mathbf{b}^2-4\mathbf{a}\mathbf{c}$ is a square in $F^\times$, and let
\begin{equation}\label{splitstandardSeq}
 S'=\mat{}{1/2}{1/2}{}.
\end{equation}
Then we can take $\lambda = 1$ and $ A=\frac1{\sqrt{\mathbf{d}}}\mat{1}{-2\mathbf{c}}{-\frac1{2\mathbf{c}}(\mathbf{b}-\sqrt{\mathbf{d}})}{\;\mathbf{b}+\sqrt{\mathbf{d}}}$ above to get $S'=\,^t\!ASA$. The torus $T'(F)$ is explicitly given by 
\begin{equation}\label{splitstandardTFeq}
 T'(F)=\{\mat{a}{}{}{d}:\:a,d\in F^\times\}.
\end{equation}
If $B\in\mathcal{S}(\Lambda,\theta,P_1)$ and $B$ is a $T_{1,0}$ eigenfunction (see (\ref{T10defeq})) with non-zero eigenvalue then an explicit computation shows that
\begin{equation}\label{BBpat1eq}
 B'(1)\neq0\qquad\Longleftrightarrow\qquad B(1)\neq0.
\end{equation}
See \cite{Pitale-Schmidt-preprint-2012} for details.
\subsubsection*{The Iwahori-spherical representations of $\GSp_4(F)$ and their Bessel models}
Table \ref{Iwahoritable} below is a reproduction of Table A.15 of \cite{NF}. It lists all the irreducible, admissible representations of $\GSp_4(F)$ that have a non-zero fixed vector under the Iwahori subgroup $I$, in a notation borrowed from \cite{SallyTadic1993}. 
In Table \ref{Iwahoritable}, all characters are assumed to be unramified. 
Further, Table \ref{Iwahoritable} lists the dimensions of the spaces of fixed vectors under the various parahoric subgroup; every parahoric subgroup is conjugate to exactly one of $K$, $P_{02}$, $P_2$, $P_1$ or $I$. In this paper we are interested in the representations that have a non-zero $P_1$-invariant vector. Except for the one-dimensional representations, these are the following:
\begin{itemize}
 \item I, IIb, IIIb, IVd, Vd and VId. These are the \emph{spherical} representations, meaning they have a $K$-invariant vecor.
 \item IIa, IVc, Vb, VIa and VIb. These are non-spherical representations that have a one-dimensional space of $P_1$-fixed vectors. Note that Vc is simply a twist of Vb by the character $\xi$, so we will not list it separately.
 \item IIIa and IVb. These are the non-spherical representations that have a two-dimensional space of $P_1$-fixed vectors.
\end{itemize}

\begin{table}
\caption[Iwahori-spherical representations]{The Iwahori-spherical representations of $\GSp_4(F)$ and the dimensions of their spaces of fixed vectors under the parahoric subgroups. Also listed are the conductor $a(\pi)$ and the value of the $\varepsilon$-factor at $1/2$. The symbol $\nu$ stands for the absolute value on $F^\times$, normalized such that $\nu(\varpi)=q^{-1}$, and $\xi$ stands for the non-trivial, unramified, quadratic character of $F^\times$.} \label{Iwahoritable}
$$\renewcommand{\arraystretch}{1.3}
 \begin{array}{cccccccccc}
  \midrule
   &&\pi&a(\pi)&\varepsilon(1/2,\pi)
    &\begin{minipage}{5.6ex}\begin{center}$K$\end{center}\end{minipage}
    &\begin{minipage}{5.6ex}\begin{center}$P_{02}$\end{center}\end{minipage}
    &\begin{minipage}{5.6ex}\begin{center}$P_2$\end{center}\end{minipage}
    &\begin{minipage}{5.6ex}\begin{center}$P_1$\end{center}\end{minipage}
    &\begin{minipage}{5.6ex}\begin{center}$I$\end{center}\end{minipage}\\
 \toprule
  {\rm I}&&\chi_1\times\chi_2\rtimes\sigma\quad
   \mbox{(irreducible)}&0&1&1&
   \renewcommand{\arraystretch}{0.5}\begin{array}[t]{c} 2 \\ {\scriptscriptstyle +-} \end{array} &4&
   \renewcommand{\arraystretch}{0.5}\begin{array}[t]{c} 4 \\ {\scriptscriptstyle ++}\\ {\scriptscriptstyle --}  \end{array} 
    &\!\!\!\renewcommand{\arraystretch}{0.5}\begin{array}[t]{c} 8 \\ {\scriptscriptstyle ++++}\\ {\scriptscriptstyle ----}\end{array}\!\!\! \\\midrule
  {\rm II}&{\rm a}&\chi\St_{\GL(2)}\rtimes\sigma&1&-(\sigma\chi)(\varpi)&0&
   \renewcommand{\arraystretch}{0.5}\begin{array}[t]{c} 1 \\ {\scriptscriptstyle -}\end{array}&2&\renewcommand{\arraystretch}{0.5}\begin{array}[t]{c} 1 \\ {\scriptscriptstyle -} \end{array}
   &\!\!\!\renewcommand{\arraystretch}{0.5}\begin{array}[t]{c} 4 \\ {\scriptscriptstyle +---} \end{array}\!\!\!
   \\
 \cmidrule{2-10}
  &{\rm b}&\chi\triv_{\GL(2)}\rtimes\sigma&0&1&1
  &\renewcommand{\arraystretch}{0.5}\begin{array}[t]{c} 1 \\ {\scriptscriptstyle +} \end{array}&2&\renewcommand{\arraystretch}{0.5}\begin{array}[t]{c} 3 \\ {\scriptscriptstyle ++-} \end{array}
  &\!\!\!\renewcommand{\arraystretch}{0.5}\begin{array}[t]{c} 4 \\ {\scriptscriptstyle +++-} \end{array}\!\!\!\\
 \midrule
  {\rm III}&{\rm a}&\chi\rtimes\sigma\St_{\GSp(2)}&2&1
   &0&0&1&\renewcommand{\arraystretch}{0.5}\begin{array}[t]{c} 2 \\ {\scriptscriptstyle +-} \end{array}
   &\!\!\!\renewcommand{\arraystretch}{0.5}\begin{array}[t]{c} 4 \\ {\scriptscriptstyle ++--} \end{array}\!\!\!\\
 \cmidrule{2-10}
  &{\rm b}&\chi\rtimes\sigma\triv_{\GSp(2)}
   &0&1&1&\renewcommand{\arraystretch}{0.5}\begin{array}[t]{c} 2 \\ {\scriptscriptstyle +-} \end{array}&3&
   \renewcommand{\arraystretch}{0.5}\begin{array}[t]{c} 2 \\ {\scriptscriptstyle +-} \end{array}
   &\!\!\!\renewcommand{\arraystretch}{0.5}\begin{array}[t]{c} 4 \\ {\scriptscriptstyle ++--} \end{array}\!\!\!\\
 \midrule
  {\rm IV}&{\rm a}&\sigma\St_{\GSp(4)}&3&-\sigma(\varpi)&0&0&0&0&
   \renewcommand{\arraystretch}{0.5}\begin{array}[t]{c} 1 \\ {\scriptscriptstyle -} \end{array}\\
 \cmidrule{2-10}
  &{\rm b}&L(\nu^2,\nu^{-1}\sigma\St_{\GSp(2)})&2&1&0&0&1&
   \renewcommand{\arraystretch}{0.5}\begin{array}[t]{c} 2 \\ {\scriptscriptstyle +-} \end{array}&\renewcommand{\arraystretch}{0.5}\begin{array}[t]{c} 3 \\ {\scriptscriptstyle ++-} \end{array}\\
 \cmidrule{2-10}
  &{\rm c}&L(\nu^{3/2}\St_{\GL(2)},\nu^{-3/2}\sigma)&1&-\sigma(\varpi)&0&
   \renewcommand{\arraystretch}{0.5}\begin{array}[t]{c} 1 \\ {\scriptscriptstyle -} \end{array}&2&\renewcommand{\arraystretch}{0.5}\begin{array}[t]{c} 1 \\ {\scriptscriptstyle -} \end{array}&\renewcommand{\arraystretch}{0.5}\begin{array}[t]{c} 3 \\ {\scriptscriptstyle +--} \end{array}\\
 \cmidrule{2-10}
  &{\rm d}&\sigma\triv_{\GSp(4)}&0&1&1&
  \renewcommand{\arraystretch}{0.5}\begin{array}[t]{c} 1 \\ {\scriptscriptstyle +} \end{array}&1&\renewcommand{\arraystretch}{0.5}\begin{array}[t]{c} 1 \\ {\scriptscriptstyle +} \end{array}&\renewcommand{\arraystretch}{0.5}\begin{array}[t]{c} 1 \\ {\scriptscriptstyle +} \end{array}\\
 \midrule
  {\rm V}&{\rm a}&\delta([\xi,\nu\xi],\nu^{-1/2}\sigma)&2&-1
   &0&0&1&0&\renewcommand{\arraystretch}{0.5}\begin{array}[t]{c} 2 \\ {\scriptscriptstyle +-} \end{array}\\
 \cmidrule{2-10}
  &{\rm b}&L(\nu^{1/2}\xi\St_{\GL(2)},\nu^{-1/2}\sigma)
   &1&\sigma(\varpi)&0&\renewcommand{\arraystretch}{0.5}\begin{array}[t]{c} 1 \\ {\scriptscriptstyle +} \end{array}&1&\renewcommand{\arraystretch}{0.5}\begin{array}[t]{c} 1 \\ {\scriptscriptstyle +} \end{array}&\renewcommand{\arraystretch}{0.5}\begin{array}[t]{c} 2 \\ {\scriptscriptstyle ++} \end{array}\\
 \cmidrule{2-10}
  &{\rm c}&\!\!L(\nu^{1/2}\xi\St_{\GL(2)},\xi\nu^{-1/2}\sigma)\!\!
   &1&-\sigma(\varpi)&0&\renewcommand{\arraystretch}{0.5}\begin{array}[t]{c} 1 \\ {\scriptscriptstyle -} \end{array}&1&\renewcommand{\arraystretch}{0.5}\begin{array}[t]{c} 1 \\ {\scriptscriptstyle -} \end{array}&\renewcommand{\arraystretch}{0.5}\begin{array}[t]{c} 2 \\ {\scriptscriptstyle --} \end{array}\\
 \cmidrule{2-10}
  &{\rm d}&L(\nu\xi,\xi\rtimes\nu^{-1/2}\sigma)&0&1
   &1&0&1&\renewcommand{\arraystretch}{0.5}\begin{array}[t]{c} 2 \\ {\scriptscriptstyle +-} \end{array}&\renewcommand{\arraystretch}{0.5}\begin{array}[t]{c} 2 \\ {\scriptscriptstyle +-} \end{array}\\
 \midrule
  {\rm VI}&{\rm a}&\tau(S,\nu^{-1/2}\sigma)&2&1&0&0&1&\renewcommand{\arraystretch}{0.5}\begin{array}[t]{c} 1 \\ {\scriptscriptstyle -} \end{array}&\renewcommand{\arraystretch}{0.5}\begin{array}[t]{c} 3 \\ {\scriptscriptstyle +--} \end{array}\\
 \cmidrule{2-10}
  &{\rm b}&\tau(T,\nu^{-1/2}\sigma)&2&1&0&0&0&
   \renewcommand{\arraystretch}{0.5}\begin{array}[t]{c} 1 \\ {\scriptscriptstyle +} \end{array}
  &\renewcommand{\arraystretch}{0.5}\begin{array}[t]{c} 1 \\ {\scriptscriptstyle +} \end{array}\\
 \cmidrule{2-10}
  &{\rm c}&L(\nu^{1/2}\St_{\GL(2)},\nu^{-1/2}\sigma)
   &1&-\sigma(\varpi)&0&\renewcommand{\arraystretch}{0.5}\begin{array}[t]{c} 1 \\ {\scriptscriptstyle -} \end{array}&1&0&\renewcommand{\arraystretch}{0.5}\begin{array}[t]{c} 1 \\ {\scriptscriptstyle -} \end{array}\\
 \cmidrule{2-10}
  &{\rm d}&L(\nu,1_{F^\times}\rtimes\nu^{-1/2}\sigma)
   &0&1&1&\renewcommand{\arraystretch}{0.5}\begin{array}[t]{c} 1 \\ {\scriptscriptstyle +} \end{array}&2&\renewcommand{\arraystretch}{0.5}\begin{array}[t]{c} 2 \\ {\scriptscriptstyle +-} \end{array}&\renewcommand{\arraystretch}{0.5}\begin{array}[t]{c} 3 \\ {\scriptscriptstyle ++-} \end{array}\\
 \toprule
 \end{array}
$$
\end{table}

Let $\pi$ be any irreducible, admissible representation of $\GSp_4(F)$. Let $S$ be a matrix as in (\ref{Sxidefeq}), and $\theta$ the associated character of $U(F)$; see (\ref{thetadefeq}). Given this data, one may ask for which characters $\Lambda$ of the torus $T(F)$, defined in (\ref{TFdefeq}) and embedded into $\GSp_4(F)$ via (\ref{TFembeddingeq}), the representation $\pi$ admits a $(\Lambda,\theta)$-Bessel model. This question can be answered, based on results from \cite{PraTak2011}. We have listed the data relevant for our current purposes in Table \ref{besselmodelstable}. Note that, in this table, the characters $\chi,\chi_1,\chi_2,\sigma,\xi$ are not necessarily assumed to be unramified, i.e., these results hold for all Borel-induced representations. For the split case $L=F\oplus F$ in Table \ref{besselmodelstable}, it is assumed that the matrix $S$ is the one in (\ref{splitstandardSeq}). The resulting torus $T(F)$ is given in (\ref{splitstandardTFeq}); embedded into $\GSp_4(F)$ it consists of all matrices ${\rm diag}(a,b,b,a)$ with $a,b$ in $F^\times$.

\begin{table}
\caption[Bessel models of Borel-induced representations]{The Bessel models of the irreducible, admissible representations of $\GSp_4(F)$ that can be obtained via induction from the Borel subgroup. The symbol $N$ stands for the norm map from $L^\times\cong T(F)$ to $F^\times$.} \label{besselmodelstable}
$$\renewcommand{\arraystretch}{1.1}
 \begin{array}{ccccc}
  \toprule
  &&\text{representation}&
  \multicolumn{2}{c}{(\Lambda,\theta)-\text{Bessel model existence condition}}
  \\
  \cmidrule{4-5}
  &&&L=F\oplus F&L/F\text{ a field extension}\\
  \toprule
  {\rm I}&& \chi_1 \times \chi_2 \rtimes \sigma\ 
  \mathrm{(irreducible)}&\text{all }\Lambda&\text{all }\Lambda\\
  \midrule
  \mbox{II}&\mbox{a}&\chi \St_{\GL(2)} \rtimes \sigma&
   \text{all }\Lambda&\Lambda\neq(\chi\sigma)\circ N \\
  \cmidrule{2-5}
  &\mbox{b}&\chi \triv_{\GL(2)} \rtimes \sigma
   &\Lambda=(\chi\sigma)\circ N &\Lambda=(\chi\sigma)\circ N \\
  \midrule
  \mbox{III}&\mbox{a}&\chi \rtimes \sigma \St_{\GSp(2)}&\text{all }\Lambda
   &\text{all }\Lambda\\\cmidrule{2-5}
  &\mbox{b}&\chi \rtimes \sigma \triv_{\GSp(2)}
   &\Lambda({\rm diag}(a,b,b,a))=&\text{---}\\
  &&&\chi(a)\sigma(ab)\text{ or }\chi(b)\sigma(ab)&\\
  \midrule
  \mbox{IV}&\mbox{a}&\sigma\St_{\GSp(4)}&\text{all }\Lambda&
  \Lambda\neq\sigma\circ N \\\cmidrule{2-5}
  &\mbox{b}&L(\nu^2,\nu^{-1}\sigma\St_{\GSp(2)})&\Lambda=\sigma\circ N 
   &\Lambda=\sigma\circ N \\\cmidrule{2-5}
  &\mbox{c}&L(\nu^{3/2}\St_{\GL(2)},\nu^{-3/2}\sigma)
   &\Lambda({\rm diag}(a,b,b,a))=&\text{---}\\
  &&&\nu(ab^{-1})\sigma(ab)\text{ or }\nu(a^{-1}b)\sigma(ab)&\\\cmidrule{2-5}
  &\mbox{d}&\sigma\triv_{\GSp(4)}&\text{---}&\text{---}\\
  \midrule
  \mbox{V}&\mbox{a}&\delta([\xi,\nu \xi], \nu^{-1/2} \sigma)&\text{all }\Lambda
   &\Lambda\neq\sigma\circ N ,\:\Lambda\neq(\xi\sigma)\circ N \\\cmidrule{2-5}
  &\mbox{b}&L(\nu^{1/2}\xi\St_{\GL(2)},\nu^{-1/2} \sigma)&\Lambda=\sigma\circ N 
   &\Lambda=\sigma\circ N ,\:\Lambda\neq(\xi\sigma)\circ N \\\cmidrule{2-5}
  &\mbox{c}&L(\nu^{1/2}\xi\St_{\GL(2)},\xi\nu^{-1/2} \sigma)
   &\Lambda=(\xi\sigma)\circ N 
   &\Lambda\neq\sigma\circ N ,\:\Lambda=(\xi\sigma)\circ N \\\cmidrule{2-5}
  &\mbox{d}&L(\nu\xi,\xi\rtimes\nu^{-1/2}\sigma)&\text{---}
   &\Lambda=\sigma\circ N ,\:\Lambda=(\xi\sigma)\circ N \\
  \midrule
  \mbox{VI}&\mbox{a}&\tau(S, \nu^{-1/2}\sigma)&\text{all }\Lambda
   &\Lambda\neq\sigma\circ N \\\cmidrule{2-5}
  &\mbox{b}&\tau(T, \nu^{-1/2}\sigma)&\text{---}&\Lambda=\sigma\circ N \\\cmidrule{2-5}
  &\mbox{c}&L(\nu^{1/2}\St_{\GL(2)},\nu^{-1/2}\sigma)
   &\Lambda=\sigma\circ N&\text{---}\\\cmidrule{2-5}
  &\mbox{d}&L(\nu,1_{F^\times}\rtimes\nu^{-1/2}\sigma)&
   \Lambda=\sigma\circ N&\text{---}\\
  \toprule
 \end{array}
$$
\end{table}
\section{Hecke eigenvalues}\label{heckeeigenvaluessec}
For integers $l$ and $m$, let $h(l,m)$ be as in (\ref{hlmdefeq}). Let $(\pi,V)$ be a smooth representation of $\GSp_4(F)$ for which $Z(\OF)$ acts trivially. We define two endomorphisms of $V$ by the formulas
\begin{align}
 \label{T10defeq}T_{1,0}v&=\frac1{{\rm vol}( P_1)}\int\limits_{ P_1h(1,0) P_1}\pi(g)v\,dg,\\
 \label{T01defeq}T_{0,1}v&=\frac1{{\rm vol}( P_1)}\int\limits_{ P_1h(0,1) P_1}\pi(g)v\,dg
\end{align}
and the \emph{Atkin-Lehner element}
\begin{equation}\label{ALdefeq}
 \eta=\begin{bmatrix}&&&-1\\&&1\\&\varpi\\-\varpi\end{bmatrix}=s_2s_1s_2\begin{bmatrix}\varpi\\&-\varpi\\&&1\\&&&-1\end{bmatrix}.
\end{equation}
Evidently, $T_{1,0}$ and $T_{0,1}$ induce endomorphisms of the subspace of $V$ consisting of $P_1$-invariant vectors. Table \ref{Iheckeeigenvaluestable} gives the eigenvalues of $T_{1,0}$ and $T_{0,1}$ on the space of $P_1$-invariant vectors for the representations in Table \ref{Iwahoritable} which have non-zero $P_1$-invariant vectors, but no non-zero $K$-invariant vectors. 
We will not give all the details of the eigenvalue calculation, since the method is similar to the one employed in \cite{schm2005}. See \cite{Pitale-Schmidt-preprint-2012} for details. Table \ref{satakenotationtable} explains the notation. 

\begin{table}[!htb]
\caption[Eigenvalues of some Iwahori-Hecke operators]{Eigenvalues of the operators $T_{1,0}$, $T_{0,1}$ and $\eta$ on spaces of $P_1$-invariant vectors in irreducible representations. } \label{Iheckeeigenvaluestable}
 $$
  \begin{array}{cccccc}
   \text{type}&\dim&T_{1,0}&T_{0,1}&\eta\\
   \toprule
   {\rm IIa}&1&\alpha\gamma q&\alpha^2\gamma^2(\alpha+\alpha^{-1})q^{3/2}&-\alpha\gamma\\
   {\rm IIIa}&2&\alpha\gamma q,\:\gamma q&\alpha\gamma^2(\alpha q+1)q,\:\alpha\gamma^2(\alpha^{-1}q+1)q&\pm\sqrt{\alpha}\gamma\\
   {\rm IVb}&2&\gamma,\:\gamma q^2&\gamma^2(q+1),\:\gamma^2q(q^3+1)&\pm\gamma\\
   {\rm IVc}&1&\gamma q&\gamma^2(q^3+1)&-\gamma\\
   {\rm Vb}&1&-\gamma q&-\gamma^2q(q+1)&\gamma\\
   {\rm VIa}&1&\gamma q&\gamma^2q(q+1)&-\gamma\\
   {\rm VIb}&1&\gamma q&\gamma^2q(q+1)&\gamma
  \end{array}
 $$
\end{table}
We make one more comment on the eigenvalues in Table \ref{Iheckeeigenvaluestable}, concerning the representations where the space of $P_1$ invariant vectors is two-dimensional. One can verify that the operators $T_{1,0}$ and $T_{0,1}$ commute, so that there exists a basis of common eigenvectors. The ordering for types IIIa and IVb in Table \ref{Iheckeeigenvaluestable} is such that \emph{the first eigenvalue for $T_{1,0}$ corresponds to the first eigenvalue for $T_{0,1}$, and the second eigenvalue for $T_{1,0}$ corresponds to the second eigenvalue for $T_{0,1}$}.
\begin{table}[!htb]
\caption[Notation for Sataka parameters]{Notation for Satake parameters for those Iwahori-spherical representations listed in Table \ref{Iheckeeigenvaluestable}. The ``restriction'' column reflects the fact that certain charcaters $\chi$ are not allowed in type IIa and IIIa representations. The last column shows the central character of the representation.} \label{satakenotationtable}
 $$
  \begin{array}{ccccccc}
   \text{type}&\text{representation}&\text{parameters}&\text{restrictions}&\text{c.c.}\\
   \toprule
   {\rm IIa}&\chi\St_{\GL(2)}\rtimes\sigma&\alpha=\chi(\varpi),\;\gamma=\sigma(\varpi)&\alpha^2\neq q^{\pm1},\:\alpha\neq q^{\pm3/2}&\chi^2\sigma^2\\
   {\rm IIIa}&\chi\rtimes\sigma\St_{\GSp(2)}&\alpha=\chi(\varpi),\;\gamma=\sigma(\varpi)&\alpha\neq1,\:\alpha\neq q^{\pm2}&\chi\sigma^2\\
   {\rm IVb}&L(\nu^2,\nu^{-1}\sigma\St_{\GSp(2)})&\gamma=\sigma(\varpi)&&\sigma^2\\
   {\rm IVc}&L(\nu^{3/2}\St_{\GL(2)},\nu^{-3/2}\sigma)&\gamma=\sigma(\varpi)&&\sigma^2\\
   {\rm Vb}&L(\nu^{1/2}\xi\St_{\GL(2)},\nu^{-1/2}\sigma)&\gamma=\sigma(\varpi)&&\sigma^2\\
   {\rm VIa}&\tau(S,\nu^{-1/2}\sigma)&\gamma=\sigma(\varpi)&&\sigma^2\\
   {\rm VIb}&\tau(T,\nu^{-1/2}\sigma)&\gamma=\sigma(\varpi)&&\sigma^2
  \end{array}
 $$
\end{table}
\section{Double cosets, an integration formula and automatic vanishing}
Let $\mathbf{a},\mathbf{b},\mathbf{c}$ be as in Sect.\ \ref{besselsubgroupsec}, subject to the conditions (\ref{standardassumptions}). Let $T(F)$ be the subgroup of $\GL_2(F)$ defined in (\ref{TFdefeq}). By \cite{Sug1985}, Lemma 2-4, there is a disjoint decomposition
\begin{equation}\label{toricdecompositioneq}
 \GL_2(F)=\bigsqcup_{m=0}^\infty T(F)\mat{\varpi^m}{}{}{1}\GL_2(\OF)
\end{equation}
(here it is important that our assumptions (\ref{standardassumptions}) on $\mathbf{a},\mathbf{b},\mathbf{c}$ are in force; for example, (\ref{toricdecompositioneq}) would obviously be wrong for $\mathbf{a}=\mathbf{c}=0$). The following two lemmas provide further decompositions for the group $\GL_2(\OF)$. We will use the notations
\begin{equation}\label{Gamma0pdefeq}
 \Gamma_0(\p)=\GL_2(\OF)\cap\mat{\OF}{\OF}{\p}{\OF}\qquad\text{and}\qquad
 \Gamma^0(\p)=\GL_2(\OF)\cap\mat{\OF}{\p}{\OF}{\OF}.
\end{equation}
\begin{lemma}\label{TOGL2OGammaplemma}
 \begin{enumerate}
  \item In the inert case $\big(\frac L\p\big)=-1$,
   \begin{equation}\label{TOGL2OGammaplemmaeq1}
    \GL_2(\OF)=T(\OF)\mat{}{1}{1}{}\Gamma_0(\p)=T(\OF)\Gamma^0(\p).
   \end{equation}
  \item In the ramified case $\big(\frac L\p\big)=0$,
   \begin{equation}\label{TOGL2OGammaplemmaeq2}
    \GL_2(\OF)=T(\OF)\mat{1}{}{u_0}{1}\Gamma_0(\p)\;\sqcup\;T(\OF)\mat{}{1}{1}{}\Gamma_0(\p)=T(\OF)\mat{1}{}{u_0}{1}\mat{}{1}{1}{}\Gamma^0(\p)\;\sqcup\;T(\OF)\Gamma^0(\p),
   \end{equation}
   with $u_0$ as in (\ref{u1u2defeq}). We have $T(\OF)\mat{1}{}{u_0}{1}\Gamma_0(\p)=\mat{1}{}{u_0}{1}\Gamma_0(\p)$.
  \item In the split case $\big(\frac L\p\big)=1$,
   \begin{align}\label{TOGL2OGammaplemmaeq3}
    \GL_2(\OF)&=T(\OF)\mat{1}{}{u_1}{1}\Gamma_0(\p)\;\sqcup\;T(\OF)\mat{1}{}{u_2}{1}\Gamma_0(\p)\;\sqcup\;T(\OF)\mat{}{1}{1}{}\Gamma_0(\p)\nonumber\\
    &=T(\OF)\mat{1}{}{u_1}{1}\mat{}{1}{1}{}\Gamma^0(\p)\;\sqcup\;T(\OF)\mat{1}{}{u_2}{1}\mat{}{1}{1}{}\Gamma^0(\p)\;\sqcup\;T(\OF)\Gamma^0(\p).
   \end{align}
   with $u_1,u_2$ as in (\ref{u1u2defeq}). We have $T(\OF)\mat{1}{}{u_i}{1}\Gamma_0(\p)=\mat{1}{}{u_i}{1}\Gamma_0(\p)$ for $i=1,2$.
   \item  Let $m$ be a positive integer, and
 \begin{equation}\label{TOmdefeq}
  T(\OF)_m:=\mat{\varpi^{-m}}{}{}{1}T(\OF)\mat{\varpi^m}{}{}{1}\cap\GL_2(\OF).
 \end{equation}
 Then
 \begin{equation}\label{TOGL2OGammapmlemmaeq1}
  \GL_2(\OF)=T(\OF)_m\Gamma_0(\p)\sqcup T(\OF)_m\mat{}{1}{1}{}\Gamma_0(\p)=T(\OF)_m\mat{}{1}{1}{}\Gamma^0(\p)\sqcup T(\OF)_m\Gamma^0(\p).
 \end{equation}
 We have $T(\OF)_m\Gamma_0(\p)=\Gamma_0(\p)$.
 \end{enumerate}
\end{lemma}
\begin{proof}
The lemma follows by using standard representatives for $\GL_2(\OF)/\Gamma_0(\p)$ and the definition of $T(\OF)$. See \cite{Pitale-Schmidt-preprint-2012}.
\end{proof}

We turn to double coset decompositions for $\GSp_4$. For $l, m \in \Z$, let
\begin{equation}\label{hlmdefeq}
 h(l,m)=\begin{bmatrix}\varpi^{l+2m}\\&\varpi^{l+m}\\&&1\\&&&\varpi^m\end{bmatrix}.
\end{equation}
Using (\ref{toricdecompositioneq}) and the Iwasawa decomposition, it is easy to see that
\begin{equation}\label{SuganoHFdecompositioneq}
 \GSp_4(F)=\bigsqcup_{\substack{l,m\in\Z\\m\geq0}}R(F)h(l,m)K;
\end{equation}
cf.\ (3.4.2) of \cite{Fu1994}. 
%
Using Lemma \ref{TOGL2OGammaplemma}, the right coset representatives for $K/P_1$ from \cite{NF}, Lemma 5.1.1,
and employing the useful identity
\begin{equation}\label{usefulidentityeq}
 \mat{1}{}{z}{1} = \mat{1}{z^{-1}}{}{1} \mat{-z^{-1}}{}{}{-z} \mat{}{1}{-1}{} \mat{1}{z^{-1}}{}{1},
\end{equation}
which holds for $z\in F^\times$, we get the following proposition (see \cite{Pitale-Schmidt-preprint-2012}).
\begin{proposition}\label{HFRFSipdecomplemma}
 \begin{enumerate}
\item
For $m>0$,
 \begin{align}\label{HFRFSipdecomplemmaeq2}
  R(F)h(l,m)K&=R(F)h(l,m) P_1\;\sqcup\;R(F)h(l,m)s_2 P_1\nonumber\\
  &\qquad\sqcup\;R(F)h(l,m)s_1s_2 P_1\;\sqcup\;R(F)h(l,m)s_2s_1s_2\, P_1,
 \end{align}
  \item In the inert case $\big(\frac L\p\big)=-1$,
   \begin{align*}
    R(F)h(l,0)K&=R(F)h(l,0) P_1\;\sqcup\;R(F)h(l,0)s_2 P_1\;\sqcup\;R(F)h(l,0)s_2s_1s_2 P_1.
   \end{align*}
  \item In the ramified case $\big(\frac L\p\big)=0$,
   \begin{align*}
    R(F)h(l,0)K&=R(F)h(l,0) P_1\;\sqcup\;R(F)h(l,0)s_2 P_1\;\sqcup\;R(F)h(l,0)s_2s_1s_2 P_1\\
    &\qquad\sqcup\;R(F)h(l,0)\hat u_0s_1s_2\, P_1,
   \end{align*}
   where
   \begin{equation}\label{hatu0defeq}
    \hat u_0=\begin{bmatrix}1\\u_0&1\\&&1&-u_0\\&&&1\end{bmatrix}.
   \end{equation}
  \item In the split case $\big(\frac L\p\big)=1$,
   \begin{align*}
    R(F)h(l,0)K&=R(F)h(l,0) P_1\;\sqcup\;R(F)h(l,0)s_2 P_1\;\sqcup\;R(F)h(l,0)s_2s_1s_2 P_1\\
    &\qquad\sqcup\;R(F)h(l,0)\hat u_1s_1s_2\, P_1\;\sqcup\;R(F)h(l,0)\hat u_2s_1s_2\, P_1,
   \end{align*}
   where, for $i=1,2$,
   \begin{equation}\label{hatu1u2defeq}
    \hat u_i=\begin{bmatrix}1\\u_i&1\\&&1&-u_i\\&&&1\end{bmatrix}.
   \end{equation}
 \end{enumerate}
\end{proposition}
\subsection*{An integration formula}
The integration formula on $\GL_2(\OF)$ presented in the following lemma will be used in many of our Hecke operator calculations. Let $\Lambda$ be a character of $L^\times\cong T(F)$. Let
\begin{equation}\label{m0defeq}
 m_0 = {\rm min}\big\{m \geq 0\::\:\Lambda\big|_{(1+\p^m\OF_L)\cap \OF_L^\times}=1\big\}.
\end{equation}
\begin{lemma}\label{GL2integrationlemma1}
 Let $\Lambda$ be a character of $L^\times\cong T(F)$ which is trivial on $\OF^\times$, and let $m_0$ be as in (\ref{m0defeq}). Let $m$ be a non-negative integer. Let $f:\:\GL_2(\OF)\rightarrow\C$ be a function with the property
 $$
  f(tg)=\Lambda(\mat{\varpi^m}{}{}{1}t\mat{\varpi^{-m}}{}{}{1})f(g)\qquad\text{for all $t\in T(\OF)_m$}.
 $$
 Let the Haar measure on $\GL_2(\OF)$ be normalized such that the total volume is $1$.
 \begin{enumerate}
  \item Assume that $f$ is right invariant under $\Gamma_0(\p)$. Then
   $$
    \int\limits_{\GL_2(\OF)}f(g)\,dg=\begin{cases}
     0&\text{if }m<m_0,\\[1ex]
     \displaystyle\renewcommand{\arraystretch}{1.0}\frac1{q+1}(f(1)+qf(\mat{}{1}{1}{}))&\text{if }m\geq {\rm max}(m_0,1),\\[3ex]
     \displaystyle\renewcommand{\arraystretch}{1.0}f(\mat{}{1}{1}{})&\text{if }m=m_0=0,\;\big(\frac L\p\big)=-1,\\[3ex]
     \displaystyle\renewcommand{\arraystretch}{1.0}\frac1{q+1}(f(\mat{1}{}{u_0}{1})+qf(\mat{}{1}{1}{}))&\text{if }m=m_0=0,\;\big(\frac L\p\big)=0,\\[3ex]
     \displaystyle\renewcommand{\arraystretch}{1.0}\frac1{q+1}(f(\mat{1}{}{u_1}{1})+f(\mat{1}{}{u_2}{1})+(q-1)f(\mat{}{1}{1}{}))\hspace{-15ex}\\[2ex]
     &\text{if }m=m_0=0,\;\big(\frac L\p\big)=1.
    \end{cases}
   $$
  \item Assume that $f$ is right invariant under $\Gamma^0(\p)$. Then
   $$
    \int\limits_{\GL_2(\OF)}f(g)\,dg=\begin{cases}
     0&\text{if }m<m_0,\\[1ex]
     \displaystyle\renewcommand{\arraystretch}{1.0}\frac1{q+1}(f(\mat{}{1}{1}{})+qf(1))&\text{if }m\geq {\rm max}(m_0,1),\\[3ex]
     \displaystyle\renewcommand{\arraystretch}{1.0}f(1)&\text{if }m=m_0=0,\;\big(\frac L\p\big)=-1,\\[3ex]
     \displaystyle\renewcommand{\arraystretch}{1.0}\frac1{q+1}(f(\mat{1}{}{u_0}{1}\mat{}{1}{1}{})+qf(1))&\text{if }m=m_0=0,\;\big(\frac L\p\big)=0,\\[3ex]
     \displaystyle\renewcommand{\arraystretch}{1.0}\frac1{q+1}(f(\mat{1}{}{u_1}{1}\mat{}{1}{1}{})+f(\mat{1}{}{u_2}{1}\mat{}{1}{1}{})+(q-1)f(1))\hspace{-21ex}\\[2ex]
     &\text{if }m=m_0=0,\;\big(\frac L\p\big)=1.
    \end{cases}
   $$
 \end{enumerate}
\end{lemma}
\begin{proof}
The lemma follows from Lemma \ref{TOGL2OGammaplemma}.
\end{proof}

\subsection*{Automatic vanishing}

Many of the cosets from Proposition \ref{HFRFSipdecomplemma} cannot be in the support of a Bessel function. The following lemma exhibits several cases of automatic vanishing. 

\begin{lemma}\label{automaticvanishinglemma}
 Let $\Lambda$ be a character of $L^\times\cong T(F)$ which is trivial on $\OF^\times$, and let $m_0$ be as in (\ref{m0defeq}). Let $B\in\mathcal{S}(\Lambda,\theta,P_1)$, and let $l$ and $m$ be integers.
 Then:
 \begin{enumerate}
  \item
   \begin{align*}
   & B(h(l,m))=B(h(l,m)s_2)=0\qquad\text{if }l<0, \\
   & B(h(l,m)s_1s_2)=B(h(l,m)s_2s_1s_2)=0\qquad\text{if }l<-1.
   \end{align*}
  \item
   \begin{align*}
   & B(h(l,m))=B(h(l,m)s_2s_1s_2)=0\qquad\text{for any $l$ and $0\leq m<m_0$},\\
   & B(h(l,m)s_2) = 0 \qquad\text{for any $l$ and $0\leq m<m_0-1$}.
   \end{align*}
  \item Assume that $m_0>0$. Then
   $$
    B(h(l,0)\hat u_is_1s_2)=0
   $$
   with $i=0$ in the ramified case and $i\in\{1,2\}$ in the split case.
  \item In the ramified case,
   $$
    B(h(l,0)\hat u_0s_1s_2)=0\qquad\text{for }l<-1.
   $$
  \item In the split case,
   $$
    B(h(l,0)\hat u_is_1s_2)=0\qquad\text{for }l<0.
   $$
 \end{enumerate}
\end{lemma}
\begin{proof}
These are standard arguments using the $P_1$-invariance of $B$ and the definition of $m_0$. 
\end{proof}

\section{The Hecke operators}\label{T10sec}
In the following lemmas we will give the value of $T_{1,0}B$ and  $T_{0,1}B$, where $B\in\mathcal{S}(\Lambda,\theta,P_1)$, on various double coset representatives from Proposition \ref{HFRFSipdecomplemma}. The main tools are the integration formulas in Lemma \ref{GL2integrationlemma1} and the useful matrix identity (\ref{usefulidentityeq}). The proofs are long and tedious and we will not present them here. See \cite{Pitale-Schmidt-preprint-2012} for details.

\begin{lemma}\label{T10actionlemma}
  Let $B\in\mathcal{S}(\Lambda,\theta,P_1)$, and let $l$ and $m$ be non-negative integers. Let $h(l,m)$ be as in (\ref{hlmdefeq}). Then the following formulas hold.
 \begin{enumerate}
  \item
   $$
    (T_{1,0}B)(h(l,m))=q^3B(h(l+1,m)).
   $$
  \item
   \begin{align*}
    &(T_{1,0}B)(h(l,m)s_2)=q^2(q-1)B(h(l+1,m))\\[1ex]
    &+\begin{cases}
     -qB(h(l-1,m+1)s_1s_2)&\text{if }m<m_0,\\[2ex]
     \displaystyle\renewcommand{\arraystretch}{1.0}q\Lambda(\varpi)B(h(l+1,m-1)s_2)+q(q-1)B(h(l-1,m+1)s_1s_2)\;\;\;\text{if }m\geq {\rm max}(m_0,1),\hspace{-28ex}\\[2ex]
     \displaystyle\renewcommand{\arraystretch}{1.0}q^2B(h(l-1,1)s_1s_2)&\text{if }m=m_0=0,\;\big(\frac L\p\big)=-1,\\[2ex]
     \displaystyle\renewcommand{\arraystretch}{1.0}q\Lambda(\varpi_L)B(h(l,0)\hat u_0s_1s_2)+q(q-1)B(h(l-1,1)s_1s_2)&\text{if }m=m_0=0,\;\big(\frac L\p\big)=0,\\[2ex]
     \displaystyle\renewcommand{\arraystretch}{1.0}q\big(\Lambda(\varpi,1)B(h(l,0)\hat u_2s_1s_2)+\Lambda(1,\varpi)B(h(l,0)\hat u_1s_1s_2)\big)\\
     \qquad+q(q-2)B(h(l-1,1)s_1s_2)&\text{if }m=m_0=0,\;\big(\frac L\p\big)=1.
    \end{cases}
   \end{align*}
  \item
   \begin{align*}
    &(T_{1,0}B)(h(l,m)s_1s_2) = q^2(q-1)B(h(l+1,m))\\
    &+\begin{cases}
     -q\Lambda(\varpi)B(h(l+1,m-1)s_2)&\text{if }m<m_0,\\[2ex]
      \displaystyle\renewcommand{\arraystretch}{1.0}q^2B(h(l-1,m+1)s_1s_2)\hspace{-2ex}&\text{if }m\geq {\rm max}(m_0,1),\\[2ex]
      \displaystyle\renewcommand{\arraystretch}{1.0}q(q+1)B(h(l-1,1)s_1s_2)-q\Lambda(\varpi)B(h(l+1,-1)s_2)&\text{if }m=m_0=0,\;\big(\frac L\p\big)=-1,\\[2ex]
      \displaystyle\renewcommand{\arraystretch}{1.0}q\Lambda(\varpi_L)B(h(l,0)\hat u_0s_1s_2)+q^2B(h(l-1,1)s_1s_2)\\[1ex]
      \qquad-q\Lambda(\varpi)B(h(l+1,-1)s_2)&\text{if }m=m_0=0,\;\big(\frac L\p\big)=0,\\[3ex]
      \displaystyle\renewcommand{\arraystretch}{1.0}q\Lambda(1,\varpi)B(h(l,0)\hat u_1s_1s_2)+q\Lambda(\varpi,1)B(h(l,0)\hat u_2s_1s_2)+q(q-1)B(h(l-1,1)s_1s_2)\hspace{-27ex}\\[2ex]
      \qquad-q\Lambda(\varpi)B(h(l+1,-1)s_2)&\text{if }m=m_0=0,\;\big(\frac L\p\big)=1.
     \end{cases}
    \end{align*}
  \item
   \begin{align*}
    &(T_{1,0}B)(h(l,m)s_2s_1s_2)=q^2(q-1)B(h(l+1,m))+\Lambda(\varpi)B(h(l-1,m)s_2s_1s_2)\\
    &+\begin{cases}
     0&\text{if }m<m_0,\\[1ex]
     \displaystyle\renewcommand{\arraystretch}{1.0}q(q-1)B(h(l-1,m+1)s_1s_2)+(q-1)\Lambda(\varpi)B(h(l+1,m-1)s_2)\hspace{-12ex}\\
       &\text{if }m\geq {\rm max}(m_0,1),\\[3ex]
     \displaystyle\renewcommand{\arraystretch}{1.0}(q^2-1)B(h(l-1,1)s_1s_2)&\text{if }m=m_0=0,\;\big(\frac L\p\big)=-1,\\[3ex]
     \displaystyle\renewcommand{\arraystretch}{1.0}(q-1)\Lambda(\varpi_L)B(h(l,0)\hat u_0s_1s_2)+q(q-1)B(h(l-1,1)s_1s_2)\hspace{-30ex}\\
      &\text{if }m=m_0=0,\;\big(\frac L\p\big)=0,\\[3ex]
     \displaystyle\renewcommand{\arraystretch}{1.0}(q-1)\big(\Lambda(1,\varpi)B(h(l,0)\hat u_1s_1s_2)+\Lambda(\varpi,1)B(h(l,0)\hat u_2s_1s_2)\big)\hspace{-20ex}\\
      \qquad+(q-1)^2B(h(l-1,1)s_1s_2)&\text{if }m=m_0=0,\;\big(\frac L\p\big)=1.
    \end{cases}
   \end{align*}
 \end{enumerate}
\end{lemma}

\begin{lemma}\label{T10u0u1u2lemma}
 Let $\Lambda$ be an unramified character of $L^\times$, and let $l\geq-1$ be an integer. Then
 \begin{enumerate}
  \item In the ramified case $\big(\frac L\p\big)=0$, for all integers $l\geq-1$,
   $$
    (T_{1,0}B)(h(l,0)\hat u_0s_1s_2)=\begin{cases}
     -q^2B(1)&\text{if }l=-1,\\
     q^2(q-1)B(h(l+1,0))+q^2B(h(l-1,1)s_1s_2)&\text{if }l\geq0.
                                  \end{cases}
   $$
  \item In the split case $\big(\frac L\p\big)=1$, for all integers $l\geq0$,
   \begin{align*}
    (T_{1,0}B)(h(l,0)\hat u_1s_1s_2)&=q^2(q-1)B(h(l+1,0))+q(q-1)B(h(l-1,1)s_1s_2)\\
    &\hspace{30ex}+q\Lambda(1,\varpi)B(h(l,0)\hat u_1s_1s_2),\\
    (T_{1,0}B)(h(l,0)\hat u_2s_1s_2)&=q^2(q-1)B(h(l+1,0))+q(q-1)B(h(l-1,1)s_1s_2)\\
    &\hspace{30ex}+q\Lambda(\varpi,1)B(h(l,0)\hat u_2s_1s_2).
   \end{align*}
 \end{enumerate}
\end{lemma}

\begin{lemma}\label{T01actionlemma}
 Let $B\in\mathcal{S}(\Lambda,\theta,P_1)$, and let $l$ and $m$ be non-negative integers. Let $h(l,m)$ be as in (\ref{hlmdefeq}). Then the following formulas hold.
 \begin{enumerate}
  \item
   $$
    (T_{0,1}B)(h(l,m))=\begin{cases}
     0&\text{if }m<m_0,\\[1ex]
     \displaystyle\renewcommand{\arraystretch}{1.0}q^3\Lambda(\varpi)B(h(l+2,m-1))+q^4B(h(l,m+1))\hspace{-9ex}\\[1ex]
     &\text{if }m\geq {\rm max}(m_0,1),\hspace{-30ex}\\[2ex]
     \displaystyle\renewcommand{\arraystretch}{1.0}q^3(q+1)B(h(l,1))&\text{if }m=m_0=0,\:\big(\frac L\p\big)=-1,\\[2ex]
     \displaystyle\renewcommand{\arraystretch}{1.0}q^3\Lambda(\varpi_L)B(h(l+1,0))+q^4B(h(l,1))\hspace{-1.4ex}&\text{if }m=m_0=0,\:\big(\frac L\p\big)=0,\\[2ex]
     \displaystyle\renewcommand{\arraystretch}{1.0}q^3(\Lambda(\varpi,1)+\Lambda(1,\varpi))B(h(l+1,0))+q^3(q-1)B(h(l,1))\hspace{-28ex}\\[2ex]
     &\text{if }m=m_0=0,\:\big(\frac L\p\big)=1.
    \end{cases}
   $$
  \item
   \begin{align*}
    &(T_{0,1}B)(h(l,m)s_2)\\
    &=\begin{cases}
     -q^3B(h(l,m+1)s_1s_2)-\Lambda(\varpi)B(h(l-2,m+1)s_1s_2)&\text{if }m<m_0,\\[2ex]
     \displaystyle\renewcommand{\arraystretch}{1.0}q^3\Lambda(\varpi)B(h(l+2,m-1)s_2)+\Lambda(\varpi)^2B(h(l,m-1)s_2)\hspace{-30ex}\\
     \qquad+(q-1)\Lambda(\varpi)B(h(l-2,m+1)s_1s_2)+q^3(q-1)B(h(l,m+1)s_1s_2)\hspace{-30ex}\\
     \qquad&\text{if }m\geq {\rm max}(m_0,1),\\[2ex]
     \displaystyle\renewcommand{\arraystretch}{1.0}q^4B(h(l,1)s_1s_2)+q\Lambda(\varpi)B(h(l-2,1)s_1s_2)&\text{if }m=m_0=0,\;\big(\frac L\p\big)=-1,\\[2ex]
     \displaystyle\renewcommand{\arraystretch}{1.0}q^3\Lambda(\varpi_L)B(h(l+1,0)\hat u_0s_1s_2)+\Lambda(\varpi)\Lambda(\varpi_L)B(h(l-1,0)\hat u_0s_1s_2)\hspace{-30ex}\\
      \qquad+(q-1)\Lambda(\varpi)B(h(l-2,1)s_1s_2)+q^3(q-1)B(h(l,1)s_1s_2)\hspace{-30ex}\\
      \qquad&\text{if }m=m_0=0,\;\big(\frac L\p\big)=0,\\[2ex]
     \displaystyle\renewcommand{\arraystretch}{1.0}q^3(\Lambda(\varpi,1)B(h(l+1,0)\hat u_2s_1s_2)+\Lambda(1,\varpi)B(h(l+1,0)\hat u_1s_1s_2))\hspace{-40ex}\\
      \qquad+\Lambda(\varpi)\big(\Lambda(\varpi,1)B(h(l-1,0)\hat u_2s_1s_2)+\Lambda(1,\varpi)B(h(l-1,0)\hat u_1s_1s_2)\big)\hspace{-40ex}\\
      \qquad+q^3(q-2)B(h(l,1)s_1s_2)+(q-2)\Lambda(\varpi)B(h(l-2,1)s_1s_2)\hspace{-50ex}\\
      &\text{if }m=m_0=0,\;\big(\frac L\p\big)=1.
    \end{cases}\\[2ex]
    &\qquad+q^2(q-1)B(h(l,m+1))+\left\{\begin{array}{ll}
     \displaystyle\Big(\frac L\p\Big)q\Lambda(\varpi)B(1),&\text{if }l=m=0,\\[2ex]
     0&\text{if }l=0,\:m\geq1,\\[2ex]
     q(q-1)\Lambda(\varpi)B(h(l,m))&\text{if }l\geq1.\end{array}\right\}.
   \end{align*}
 \end{enumerate}
\end{lemma}

\begin{lemma}\label{T01u0u1u2lemma}
 Let $\Lambda$ be an unramified character of $L^\times$, and let $l\geq-1$ be an integer. Then
 \begin{enumerate}
  \item In the ramified case $\big(\frac L\p\big)=0$, for all integers $l\geq-1$,
   $$
    (T_{0,1}B)(h(l,0)\hat u_0s_1s_2)=\begin{cases}
     q^4B(h(-1,1)s_1s_2)-q^2\Lambda(\varpi_L)B(1)&\text{if }l=-1,\\
     q^4B(h(0,1)s_1s_2)\\
      \qquad+q^2(q-1)\Lambda(\varpi_L)B(h(1,0))-q\Lambda(\varpi)B(1)&\text{if }l=0,\\
     q^4B(h(l,1)s_1s_2)+q\Lambda(\varpi)B(h(l-2,1)s_1s_2)\\
      \qquad+q^2(q-1)\Lambda(\varpi_L)B(h(l+1,0))\\
      \qquad+\Lambda(\varpi)q(q-1)B(h(l,0)))&\text{if }l\geq1.
                                  \end{cases}
   $$
  \item In the split case $\big(\frac L\p\big)=1$, for all integers $l\geq0$,
   $$
    (T_{0,1}B)(h(l,0)\hat u_1s_1s_2)=\begin{cases}
     q^3\Lambda(1,\varpi)B(h(1,0)\hat u_1s_1s_2)+q^3(q-1)B(h(0,1)s_1s_2)\\
      \qquad+q^2(q-1)\Lambda(\varpi,1)B(h(1,0))&\text{if }l=0,\\
     q^3\Lambda(1,\varpi)B(h(l+1,0)\hat u_1s_1s_2)+q^3(q-1)B(h(l,1)s_1s_2)\hspace{-3ex}\\
      \qquad+q^2(q-1)\Lambda(\varpi,1)B(h(l+1,0))\\
      \qquad+\Lambda(\varpi)q(q-1)B(h(l,0))\\
      \qquad+\Lambda(\varpi)\Lambda(1,\varpi)B(h(l-1,0)\hat u_1s_1s_2)\\
      \qquad+(q-1)\Lambda(\varpi)B(h(l-2,1)s_1s_2)&\text{if }l\geq1.
    \end{cases}
   $$
   $$
    (T_{0,1}B)(h(l,0)\hat u_2s_1s_2)=\begin{cases}
     q^3\Lambda(\varpi,1)B(h(1,0)\hat u_2s_1s_2)+q^3(q-1)B(h(0,1)s_1s_2)\\
      \qquad+q^2(q-1)\Lambda(1,\varpi)B(h(1,0))&\text{if }l=0,\\
     q^3\Lambda(\varpi,1)B(h(l+1,0)\hat u_2s_1s_2)+q^3(q-1)B(h(l,1)s_1s_2)\hspace{-3ex}\\
      \qquad+q^2(q-1)\Lambda(1,\varpi)B(h(l+1,0))\\
      \qquad+\Lambda(\varpi)q(q-1)B(h(l,0))\\
      \qquad+\Lambda(\varpi)\Lambda(\varpi,1)B(h(l-1,0)\hat u_2s_1s_2)\\
      \qquad+(q-1)\Lambda(\varpi)B(h(l-2,1)s_1s_2)&\text{if }l\geq1.
    \end{cases}
   $$
 \end{enumerate}
\end{lemma}

\begin{lemma}\label{T01s2conslemma}
 Let $B\in\mathcal{S}(\Lambda,\theta,P_1)$ be such that $T_{1,0}B=\lambda B$ and $T_{0,1}B=\mu B$. Let $l$ and $m$ be non-negative integers. 
 \begin{itemize}
 \item For $m\geq{\rm max}(m_0,1)$, or for $m\geq m_0-1$ if $m_0>0$, we have
 \begin{align}\label{T01s2conslemmaeq1}
  &\lambda\Lambda(\varpi) B(h(l,m)s_2)-\mu qB(h(l+1,m)s_2)+\lambda q^3B(h(l+2,m)s_2)\nonumber\\
  &\qquad=q^5(q-1)B(h(l+3,m))-q^3(q-1)B(h(l+1,m+1))
 \end{align}
 
 \item For $m\geq{\rm max}(m_0,1)$, we have
 \begin{align}\label{T01s2conseq4}
 &\mu B(h(0,m)s_2)-q^2\lambda B(h(1,m)s_2)-\Lambda(\varpi)^2B(h(0,m-1)s_2)\nonumber\\
 &\qquad=q^2(q-1)B(h(0,m+1))-q^4(q-1)B(h(2,m))
\end{align}

\item For $m_0 > 0$, we have
\begin{align}\label{T01s2conseq4b}
 &\mu B(h(0,m_0)s_2)-q^2\lambda B(h(1,m_0)s_2)-\Lambda(\varpi)^2B(h(0,m_0-1)s_2)\nonumber\\
 &\qquad=q^{-2}(q-1)(\mu-\lambda^2)B(h(0,m_0))
\end{align}

\item For $l \geq 0, m_0 > 0$, we have
\begin{align}\label{T01s2conslemmaeq1b}
  &\lambda\Lambda(\varpi) B(h(l,m_0-1)s_2)-\mu qB(h(l+1,m_0-1)s_2)+\lambda q^3B(h(l+2,m_0-1)s_2)\nonumber\\
  &\qquad=-\lambda(q-1)B(h(l,m_0))
\end{align}
\end{itemize}
\end{lemma}
\begin{proof}
(\ref{T01s2conslemmaeq1}) and (\ref{T01s2conseq4}) follow from Lemma \ref{T01actionlemma} ii) and Lemma \ref{T10actionlemma} ii). (\ref{T01s2conseq4b}) follows from (\ref{T01s2conseq4}), Lemma \ref{T10actionlemma} i) and Lemma \ref{T01actionlemma} i). (\ref{T01s2conslemmaeq1b}) follows from (\ref{T01s2conslemmaeq1}), the identity $\lambda B(h(l,m_0))=q^3B(h(l+1,m_0))$ and some automatic vanishing.
\end{proof}

\section{The main tower}\label{maintowersec}
Again we consider the matrix $S$ and the associated character $\theta$ of $U(F)$; see (\ref{thetadefeq}). As usual, the assumptions (\ref{standardassumptions}) are in force. Let $\Lambda$ be a character of $T(F)$ and define the non-negative integer $m_0$ as in (\ref{m0defeq}). Let $B$ be a function in the space $\mathcal{S}(\Lambda,\theta,P_1)$. We refer to the values of $B$ at the elements $h(l,m)$, defined in (\ref{hlmdefeq}), as the \emph{main tower} (in view of Proposition \ref{HFRFSipdecomplemma}, there is also an $s_2$-tower etc). Note that $B(h(l,m))=0$ for $l<0$ or $0\leq m<m_0$ by Lemma \ref{automaticvanishinglemma}. The following result relates the values of $B(h(l,m))$ to $B(h(0,m_0))$, only assuming that $B$ is an eigenfunction for $T_{1,0}$ and $T_{0,1}$. In particular, it shows that the entire main tower is zero if $B(h(0,m_0))=0$.
\begin{proposition}\label{maintowerprop}
 Let $B\in\mathcal{S}(\Lambda,\theta,P_1)$ be an eigenfunction for $T_{1,0}$ and $T_{0,1}$ with eigenvalues $\lambda$ and $\mu$, respectively. Then
 \begin{equation}\label{main-tower-gen-l+1tol}
  B(h(l+1,m)) =\lambda q^{-3}B(h(l,m))\qquad\text{for all $l\geq0$ and $m\geq0$}.
 \end{equation}
 Furthermore, for any $l\geq0$, there is a formal identity
 $$
  Y^{-m_0}\sum_{m=m_0}^\infty B(h(l,m))Y^m =\frac{1-\kappa q^{-4}Y}{1-\mu q^{-4} Y+\lambda^2q^{-7}\Lambda(\varpi)Y^2}B(h(l,m_0)),
 $$
 where
 $$
  \kappa=\begin{cases}
          0&\text{if }m_0>0,\\[1ex]
          (q+1)^{-1}\mu&\text{if }m_0=0\text{ and }\big(\frac L\p\big)=-1,\\[1ex]
          \Lambda(\varpi_L)\lambda&\text{if }m_0=0\text{ and }\big(\frac L\p\big)=0,\\[1ex]
          (q-1)^{-1}(q\lambda(\Lambda(\varpi, 1)+\Lambda(1,\varpi))-\mu)&\text{if }m_0=0\text{ and }\big(\frac L\p\big)=1.
         \end{cases}
 $$
\end{proposition}
\begin{proof}
The relation (\ref{main-tower-gen-l+1tol}) is immediate from Lemma \ref{T10actionlemma} i). Combining (\ref{main-tower-gen-l+1tol}) with Lemma \ref{T01actionlemma} i), we get, for $l\geq0$ and $m\geq m_0$,
\begin{equation}\label{general-2-step-recursion}
 q^4B(h(l,m+2)) - \mu B(h(l,m+1)) +\lambda^2q^{-3}\Lambda(\varpi)B(h(l,m)) = 0.
\end{equation}
We multiply this by $Y^{m+2}$ and apply $\sum_{m=m_0}^\infty$ to both sides, arriving at the formal identity
$$
 \sum_{m=m_0}^\infty B(h(l,m))Y^m=\frac{(q^4-\mu Y)B(h(l,m_0))+q^4YB(h(l,m_0+1))}{q^4-\mu Y+\lambda^2q^{-3}\Lambda(\varpi)Y^2}Y^{m_0}.
$$
Setting $m=m_0$ in Lemma \ref{T01actionlemma} i) and using (\ref{main-tower-gen-l+1tol}) provides a relation between $B(h(l,m_0))$ and $B(h(l,m_0+1))$. Substituting this relation, we obtain the asserted formula.
\end{proof}
\section{Generic representations and split Bessel models}\label{genericsplitsec}
We recall some basic facts about generic representations of $\GSp_4(F)$, and refer to Sect.\ 2.6 of \cite{NF} for details. We denote by $N$ the unipotent radical of the Borel subgroup $B$, and for $c_1,c_2$ in $F^\times$, consider the character $\psi_{c_1,c_2}$ of $N(F)$ given in Sect.\,2.1 of \cite{NF}.
An irreducible, admissible representation $\pi$ of $\GSp_4(F)$ is called \emph{generic} if ${\rm Hom}_{N(F)}(\pi,\psi_{c_1,c_2})\neq0$. In this case there is an associated Whittaker model $\mathcal{W}(\pi,\psi_{c_1,c_2})$ consisting of functions $\GSp_4(F)\rightarrow\C$ that transform on the left according to $\psi_{c_1,c_2}$. For $W\in\mathcal{W}(\pi,\psi_{c_1,c_2})$, there is an associated zeta integral
\begin{equation}\label{ZsWdefeq}
 Z(s,W)=\int\limits_{F^\times}\int\limits_F
 W(\begin{bmatrix}a\\&a\\x&&1\\&&&1\end{bmatrix})|a|^{s-3/2}\,dx\,d^\times a.
\end{equation}
This integral is convergent for ${\rm Re}(s)>s_0$, where $s_0$ is independent
of $W$; see \cite{NF}, Proposition 2.6.3. 
Moreover, there exists an $L$-factor of the form
$
 L(s,\pi)=Q(q^{-s})^{-1},Q(X)\in\C[X],\;Q(0)=1,
$
such that
\begin{equation}\label{zetaLquotienteq}
 \frac{Z(s,W)}{L(s,\pi)}\in\C[q^{-s},q^s]\qquad\text{for all }
 W\in \mathcal{W}(\pi,\psi_{c_1,c_2}).
\end{equation}
We consider split Bessel models with respect to the quadratic form $S$ defined in (\ref{splitstandardSeq}).
Let $\theta$ be the corresponding character of $U(F)$. The resulting group $T(F)$, defined in (\ref{TFdefeq}), is a split torus. We think of $T(F)$ embedded into $\GSp_4(F)$ as all matrices of the form ${\rm diag}(a,b,b,a)$ with $a,b\in F^\times$. We write a character $\Lambda$ of $T(F)$ as a function $\Lambda(a,b)$. Consider the functional $f_s$ on the $\psi_{c_1,c_2}$-Whittaker model of $\pi$ given by
\begin{equation}\label{fs-defn}
 f_s(W)=\frac{Z(s,\pi(s_2)(W))}{L(s,\pi)}.
\end{equation}
By analytic continuation and the defining property of $L(s,\pi)$, this is a well-defined and non-zero functional on $\pi$ for \emph{any} value of $s$. A direct computation shows that the functional $f_s$ is a split $(\Lambda,\theta)$-Bessel functional with respect to the character $\Lambda$ given by
$
 \Lambda({\rm diag}(a,b,b,a))=|a^{-1}b|^{-s+1/2}.
$
If we assume that $\pi$ has trivial central character, then any unramified character of $T(F)$ is of this form for an appropriate $s$.
\subsection*{Zeta integrals of Siegel vectors}
Let $(\pi,V)$ be an irreducible, admissible, generic representation of $\GSp_4(F)$ with unramified central character. We assume that $V=\mathcal{W}(\pi,\psi_{c_1,c_2})$ is the Whittaker model with respect to the character $\psi_{c_1,c_2}$ of $N(F)$. For what follows we will assume that $c_1,c_2\in\OF^\times$. Recall the Siegel congruence subgroup $P_1$ and the Klingen congruence subgroup $P_2$ defined in (\ref{parahoricsubgroupsdefeq}). 

\begin{lemma}\label{simplifiedzetalemma}
 Let $(\pi,V)$ be as above, and let $W$ be an element of $V=\mathcal{W}(\pi,\psi_{c_1,c_2})$.
 \begin{enumerate}
  \item If $W$ is $P_1$-invariant, then
   $$
    Z(s,\pi(s_2)W)=\int\limits_{F^\times}W(\begin{bmatrix}a\\&a\\&&1\\&&&1\end{bmatrix}s_2)|a|^{s-3/2}\,d^\times a.
   $$
  \item If $W$ is $P_2$-invariant, then
   $$
    Z(s,W)=\int\limits_{F^\times}W(\begin{bmatrix}a\\&a\\&&1\\&&&1\end{bmatrix})|a|^{s-3/2}\,d^\times a.
   $$
 \end{enumerate}
\end{lemma}
\begin{proof}
ii) is a special case of Lemma 4.1.1 of \cite{NF}, and i) is proved similarly. 
\end{proof}

Let $W\in V$ be a $P_2$-invariant vector. Then
$$
 T_{\rm Si}W:=\sum_{g\in\GL_2(\OF)/\Gamma^0(\p)}\pi(\mat{g}{}{}{^tg^{-1}})W
$$
is $P_1$-invariant. Using Lemma \ref{simplifiedzetalemma}, the standard representatives for $\GL_2(\OF)/\Gamma^0(\p)$ and (\ref{usefulidentityeq}), one calculates that
\begin{equation}\label{zetaintegralformulaeq}
 Z(s,\pi(s_2)(T_{\rm Si}W))=q\int\limits_{F^\times}
 W(\begin{bmatrix}a\\&a\\&&1\\&&&1\end{bmatrix}s_2s_1)|a|^{s-3/2}\,d^\times a+Z(s,W).
\end{equation}
\subsection*{The IIa case}

\begin{lemma}\label{IIazetalemma}
 Let $\pi=\chi\St_{\GL(2)}\rtimes\sigma$ be a representation of type IIa with unramified $\chi$ and $\sigma$. Let $\alpha=\chi(\varpi)$ and $\gamma=\sigma(\varpi)$. We assume that $\alpha^2\gamma^2=1$, i.e., that $\pi$ has trivial central character. Let $W_0$ be a non-zero $\para(\p)$-invariant vector in the $\psi_{c_1,c_2}$-Whittaker model of $\pi$ normalzed such that $Z(s,W_0)=L(s,\pi)$. Let $\omega=-\alpha\gamma$ be the eigenvalue of the Atkin-Lehner element $\eta$ on $W_0$.
 \begin{enumerate}
  \item 
  For any $s\in\C$,
   $$
    Z(s,\pi(s_2)(T_{\rm Si}W_0))=(\omega q^{s-1/2}+1)L(s,\pi).
   $$
\item  If $\omega q^{s-1/2}+1=0$, then
   $$
    Z(s,T_{\rm Si}W_0)=(1-q^{-1})^{-1}.
   $$
 \end{enumerate}

\end{lemma}
\begin{proof}
We may assume that $s$ is in the region of convergence. Since the element $t_1 = h(1,0)^{-1}s_1s_2s_1h(1,0)$ lies in the paramodular group $\para(\p)$ we have, for any $g$ in $\GSp_4(F)$, the relation
$ \omega W_0(gs_2s_1)=W_0(gh(1,0)).$ Substituting this into (\ref{zetaintegralformulaeq}), we obtain
\begin{align*}
 Z(s,\pi(s_2)(T_{\rm Si}W_0))&=\omega q\int\limits_{F^\times}W_0(\begin{bmatrix}a\varpi\\&a\varpi\\&&1\\&&&1\end{bmatrix})|a|^{s-3/2}\,d^\times a+Z(s,W_0)\\
 &=(\omega q^{s-1/2}+1)L(s,\pi).
\end{align*}
This proves part i) of the lemma. Proof of part ii) is a long computation and we do not present it here. See \cite{Pitale-Schmidt-preprint-2012} for details.
\end{proof}

For the following proposition we continue to assume that $\pi=\chi\St_{\GL(2)}\rtimes\sigma$ is a representation of type IIa with unramified $\chi$ and $\sigma$, but we will drop the condition that $\pi$ has trivial central character. If the non-zero vector $W$ spans the space of $P_1$-invariant vectors, then we still have $\eta W=\omega W$ with $\omega=-\alpha\gamma$, but this constant is no longer necessarily $\pm1$. We consider split Bessel models with respect $S$ as in (\ref{splitstandardSeq}). We want the character $\Lambda$ of $T(F)$ to coincide on the center of $\GSp_4(F)$ with the central character of $\pi$, i.e., $\Lambda(a,a)=(\chi\sigma)^2(a)$. Since $\Lambda(\varpi,1)\Lambda(1,\varpi)=\omega^2$, we have $\Lambda(1,\varpi)=-\omega$ if and only if $\Lambda(\varpi,1)=-\omega$.

\begin{proposition}\label{IIasplitprop}
 Assume that $\pi=\chi\St_{\GL(2)}\rtimes\sigma$ is a representation of type IIa with unramified $\chi$ and $\sigma$. Let $S$, $\theta$ and $T(F)$ be as above. Let $\Lambda$ be an unramified character of $T(F)$ such that $\Lambda(a,a)=(\chi\sigma)^2(a)$. Let $B$ be a non-zero vector in the $(\Lambda,\theta)$-Bessel model of $\pi$ spanning the one-dimensional space of $P_1$-invariant vectors. Then we have
 $$\Lambda(\varpi,1)\neq-\omega\neq\Lambda(1,\varpi) \quad \Longleftrightarrow \quad B(1)\neq0.$$
\end{proposition}
\begin{proof}
After twisting by an unramified character, we may assume that $\pi$ has trivial central character. Let $W_0$ be a non-zero vector in the $\psi_{c_1,c_2}$-Whittaker model of $\pi$ spanning the one-dimensional space of $\para(\p)$-invariant vectors. By Lemma \ref{IIazetalemma}, the vector $T_{\rm Si}W_0$ is non-zero, and hence spans the one-dimensional space of $P_1$-invariant vectors. Let $f_s$ be as in (\ref{fs-defn}).
%
Note that
$
 \Lambda(\varpi,1)=-\omega=\Lambda(1,\varpi)$ if and only if $\omega q^{s-1/2}+1=0$. Then, by i) of Lemma \ref{IIazetalemma}, we have $f_s(T_{\rm Si}W_0)\neq0$ if and only if $\Lambda(\varpi,1)\neq-\omega\neq\Lambda(1,\varpi)$. The proposition now follows from the fact that $f_s(T_{\rm Si}W_0)\neq0$ is equivalent to $B(1)\neq0$ by uniqueness of Bessel models.
\end{proof}
\subsection*{The VIa case}
Let $\pi=\tau(S,\nu^{-1/2}\sigma)$ be the representation of type VIa. We assume that $\sigma$ is unramified and that $\pi$ has trivial central character, i.e., $\sigma^2=1$. This is a representation of conductor $2$
By Table A.8 of \cite{NF},
\begin{equation}\label{VIaLeq}
 L(s,\pi)=L(s,\nu^{1/2}\sigma)^2=\frac1{(1-\sigma(\varpi)q^{-1/2-s})^2}.
\end{equation}
We assume that $\pi$ is given in its $\psi_{c_1,c_2}$-Whittaker model. Let $W_0$ be the paramodular newform, i.e., a non-zero element invariant under the paramodular group of level $\p^2$; see (\ref{paradefeq}). By Theorem 7.5.4 of \cite{NF}, we can normalize $W_0$ such that
$
 Z(s,W_0)=L(s,\pi).
$
We let
\begin{equation}\label{shadowvectordefeq}
 W':=\sum_{x,y,z\in\OF/\p}\pi(\begin{bmatrix}1&y\varpi\\&1\\&x\varpi&1\\x\varpi&z\varpi&-y\varpi&1\end{bmatrix})W_0.
\end{equation}
This is a $P_2$-invariant vector; it was called the \emph{shadow of the newform} in Sect.\ 7.4 of \cite{NF}. By Proposition 7.4.8 of \cite{NF},
\begin{equation}\label{VIashadowzetaeq}
 Z(s,W')=(1-q^{-1})L(s,\pi).
\end{equation}

\begin{lemma}\label{VIazetalemma}
 For any complex number $s$,
 \begin{equation}\label{VIacalceq8}
  \frac{Z(s,\pi(s_2)(T_{\rm Si}W'))}{L(s,\pi)}=2(q-1)\sigma(\varpi)q^{-1/2+s}.
 \end{equation}
\end{lemma}
\begin{proof}
We may assume that $s$ is in the region of convergence. By (\ref{zetaintegralformulaeq}) and (\ref{VIashadowzetaeq}),
\begin{equation}\label{zetaintegralformulaeq2}
 Z(s,\pi(s_2)(T_{\rm Si}W'))=q\int\limits_{F^\times}
 W'(\begin{bmatrix}a\\&a\\&&1\\&&&1\end{bmatrix}s_2s_1)|a|^{s-3/2}\,d^\times a+(1-q^{-1})L(s,\pi).
\end{equation}
To evaluate the integral in this equation, we compute the zeta integral $Z(s,\pi(s_2s_1)W')$ in two different ways -- first, by using the local functional equation for $Z(s,W)$, and second, by using a direct computation and the fact that the representation has no $\para(\p)$-invariant vectors. The result is that
\begin{equation}\label{VIacalceq5}
 \int\limits_{F^\times}W'(\begin{bmatrix}a\\&a\\&&1\\&&&1\end{bmatrix}s_2s_1)|a|^{s-3/2}\,d^\times a=(q-1)q^{2s-1}(L(s,\pi)-1).
\end{equation}
From  (\ref{zetaintegralformulaeq2}) and (\ref{VIacalceq5}), we now get
\begin{equation}\label{VIacalceq6}
 Z(s,\pi(s_2)(T_{\rm Si}W'))=(q-1)q^{2s}(L(s,\pi)-1)+(1-q^{-1})L(s,\pi).
\end{equation}
Using the explicit form (\ref{VIaLeq}) of $L(s,\pi)$, the assertion follows. See \cite{Pitale-Schmidt-preprint-2012} for further details.
\end{proof}

%
\begin{proposition}\label{VIasplitprop}
 Assume that $\pi=\tau(S,\nu^{-1/2}\sigma)$ is a representation of type VIa with unramified $\sigma$. Let $S$ be as in (\ref{splitstandardSeq}). Let $\Lambda$ be an unramified character of $T(F)$ such that $\Lambda(a,a)=\sigma^2(a)$. Let $B$ be a non-zero vector in the $(\Lambda,\theta)$-Bessel model of $\pi$ spanning the one-dimensional space of $P_1$-invariant vectors. Then $B(1)\neq0$.
\end{proposition}
\begin{proof}
After twisting by an unramified character, we may assume that $\pi$ has trivial central character. Let $W_0$ be a non-zero vector in the $\psi_{c_1,c_2}$-Whittaker model of $\pi$ spanning the one-dimensional space of $\para(\p^2)$-invariant vectors, and let $W'$ be the $P_2$-invariant vector defined in (\ref{shadowvectordefeq}). By Lemma \ref{VIazetalemma}, the vector $T_{\rm Si}W'$ is non-zero, and hence spans the one-dimensional space of $P_1$-invariant vectors. Let $f_s$ be as in (\ref{fs-defn}).
%
By Lemma \ref{VIazetalemma}, we have $f_s(T_{\rm Si}W')\neq0$. By uniqueness of Bessel models, this is equivalent to $B(1)\neq0$.
\end{proof}
\section{The one-dimensional cases}\label{onedimsec}
In this section we will identify good test vectors for those irreducible, admissible representations of $\GSp_4(F)$ which are not spherical, but possess a one-dimensional space of $P_1$-invariant vectors. A look at Table \ref{Iwahoritable} shows that these are the Iwahori-spherical representations of type IIa, IVc, Vb, VIa and VIb (Vc is a twist of Vb and is not counted separately). 
This entire section assumes that the elements $\mathbf{a},\mathbf{b},\mathbf{c}$ satisfy the conditions (\ref{standardassumptions}). The matrix $S$, the group $T(F)\cong L^\times$ and the character $\theta$ have the usual meaning. 

\begin{lemma}\label{onedimlemma}
 Let $\Lambda$ be a character of $L^\times$, and $m_0$ as in (\ref{m0defeq}). 
 \begin{enumerate}
 \item Assume that $B\in\mathcal{S}(\Lambda,\theta,P_1)$ satisfies
 $$
  T_{1,0}B=\lambda B,\qquad T_{0,1}B=\mu B,\qquad \eta B=\omega B
 $$
 with complex numbers $\lambda,\mu,\omega$ satisfying $\lambda = -q\omega$.
\begin{itemize}
 \item If $m_0=0$, $\big(\frac L\p\big)=0$ and $\Lambda(\varpi_L)=-\omega$, then $B=0$. 
 \item If $m_0=0$, $\big(\frac L\p\big)=1$, $\Lambda(1,\varpi)=-\omega=\Lambda(\varpi,1)$ and $B(1)=0$, then
$$
      B=0\qquad\Longleftrightarrow\qquad B(\hat u_1s_1s_2)=0\qquad\Longleftrightarrow\qquad B(\hat u_2s_1s_2)=0.
     $$
\item Otherwise, 
 $$
    B=0\qquad\Longleftrightarrow\qquad B(h(0,m_0))=0.
   $$
   \end{itemize}
\item Assume that $B\in\mathcal{S}(\Lambda,\theta,P_1)$ satisfies  $\eta B=\omega B$ and 
\begin{equation}\label{VIbspecialeq}
 \sum_{c\in\OF/\p}\pi(\begin{bmatrix}1\\&1\\c&&1\\&&&1\end{bmatrix})B+\pi(s_2)B=0.
\end{equation}
Then
$$
    B=0\qquad\Longleftrightarrow\qquad B(h(0,m_0))=0.
   $$
\end{enumerate}
%
\end{lemma}
\begin{proof}
i) From Lemma \ref{T10actionlemma} and the Atkin-Lehner relations, we get
\begin{align}
\label{onedimT10eq1}&\lambda B(h(l,m))=q^3B(h(l+1,m)), \quad \text{ for } l,m \geq 0, \\
\label{onedimT10eq5a}
& (q-1)B(h(l,m))=q^2B(h(l,m)s_2)+q\omega B(h(l+1,m-1)s_2) \quad \text{ for } l \geq 0,\: m\geq\max(m_0,1).
\end{align}

Assume that $m_0>0$ and that $B(h(0,m_0))=0$. Then, by Proposition \ref{maintowerprop}, the values $B(h(l,m))$ are zero for all $l$ and $m$. Considering (\ref{onedimT10eq5a}) with $l=0$, (\ref{onedimT10eq5a}) with $l=1$, Lemma \ref{T01s2conslemma} with $l=0$, $m=m_0-1$, (\ref{T01s2conseq4}) and Lemma \ref{T01actionlemma} ii) with $l=0$, $m=m_0-1$ we get a homogeneous system of $5$ linear equations with a non-singular matrix in the variables
$$
 B(h(0,m_0-1)s_2),\:B(h(1,m_0-1)s_2),\:B(h(2,m_0-1)s_2),\:B(h(0,m_0)s_2),\:B(h(1,m_0)s_2).
$$
Hence, the above values are equal to zero.
Lemma \ref{T01s2conslemma} then implies that $B(h(l,m_0-1)s_2)=0$ for all $l\geq0$. Using (\ref{onedimT10eq5a}), it follows that $B(h(l,m)s_2)=0$ for all $l\geq0$ and $m\geq m_0-1$. 
The Atkin-Lehner relations now imply that the main tower, $s_2$-tower, $s_1s_2$-tower and $s_2s_1s_2$-tower are all zero. By Proposition \ref{HFRFSipdecomplemma} and Lemma \ref{automaticvanishinglemma}, the function $B$ is zero. This proves i) of the lemma for $m_0>0$.

Now assume that $m_0=0$.  Using Proposition \ref{maintowerprop}, Lemma \ref{automaticvanishinglemma}, (\ref{onedimT10eq5a}), and the Atkin-Lehner relations, we can show the following: If $B(1)=0$ and $B(h(l,0)s_2)=0$ for all $l\geq0$, then $B(h(l,m)w)=0$ for all $l\in\Z$, all $m\geq0$, and all $w\in\{1,s_2,s_2s_1s_2\}$, as well as $B(h(l,m)s_1s_2)=0$ for all $l\in\Z$ and all $m\geq1$.

Assume that $\big(\frac L\p\big)=-1$ (the inert case). Then, from Lemma \ref{T10actionlemma} ii),
 (\ref{onedimT10eq1}), Atkin-Lehner relations and using $\lambda=-q\omega$, we get
$ q(q+1)B(h(l,0)s_2)=(q-1)B(h(l,0))$ for $l\geq0$. Now, from the double coset representatives in Proposition \ref{HFRFSipdecomplemma}, it follows that $B=0$ if $B(1)=0$. The other cases are similar. See \cite{Pitale-Schmidt-preprint-2012} for details.

ii) Evaluating (\ref{VIbspecialeq}) at $h(l,m)s_2$, we obtain
\begin{equation}\label{VIbspecialeq2}
 q B(h(l,m)s_2)=-B(h(l,m))
\end{equation}
for all $l,m\geq0$. Let $i=0$ in the ramified case and $i=1$ or $i=2$ in the split case. Evaluating (\ref{VIbspecialeq}) at $h(l,0)\hat u_is_1s_2s_1s_2$ leads to
\begin{equation}\label{VIbspecialeq3}
 q B(h(l,0)s_2s_1s_2)=-B(h(l,0)\hat u_is_1s_2)
\end{equation}
for $l\geq-1$; observe here the defining property (\ref{u1u2defeq}) of the elements $u_i$. Assume that $B(h(0,m_0))$, and therefore, by Proposition \ref{maintowerprop}, the whole main tower, is zero. Then, by (\ref{VIbspecialeq2}) and Lemma \ref{automaticvanishinglemma}, the entire $s_2$-tower is also zero. By the Atkin-Lehner relations, the $s_1s_2$ and $s_2s_1s_2$ towers are also zero. In the inert case, in view of the double cosets given in Proposition \ref{HFRFSipdecomplemma}, and v) of Lemma \ref{automaticvanishinglemma}, it follows that $B=0$. Taking into account (\ref{VIbspecialeq3}), the same conclusion holds in the ramified and split cases.
\end{proof}

The condition $\lambda = -q\omega$ in i) of the above lemma is satisfied by the representations of type IIa, IVc, Vb and VIa. The representations of type VIb satisfies $\lambda = q\omega$. These representations have a one-dimensional space of $P_1$-invariant vectors, but no non-zero $P_2$-invariant vectors. Hence, if a non-zero $P_1$-invariant vector $B$ is made $P_2$-invariant by summation, the result is zero -- using standard representatives, we get (\ref{VIbspecialeq}) in ii) of the above lemma.

\subsection*{The main result}
The following theorem, which is the main result of this section, identifies good test vectors for those irreducible, admissible, infinite-dimensional representations of $\GSp_4(F)$ that have a one-dimensional space of $P_1$-invariant vectors. 
\begin{theorem}\label{onedimtheorem}
 Let $\pi$ be an irreducible, admissible, Iwahori-spherical representation of $\GSp_4(F)$ that is not spherical but has a one-dimensional space of $P_1$-invariant vectors. Let $S$ be the matrix defined in (\ref{Sxidefeq}), with $\mathbf{a},\mathbf{b},\mathbf{c}$ subject to the conditions (\ref{standardassumptions}). Let $T(F)$ be the group defined in (\ref{TFdefeq}). Let $\theta$ be the character of $U(F)$ defined in (\ref{thetadefeq}), and let $\Lambda$ be a character of $T(F)\cong L^\times$. Let $m_0$ be as in (\ref{m0defeq}). We assume that $\pi$ admits a $(\Lambda,\theta)$-Bessel model. Let $B$ be an element in this Bessel model spanning the space of $P_1$-invariant vectors.
 \begin{enumerate}
  \item Assume that $\pi$ is of type IIa. Then $B(h(0,m_0))\neq0$, except in the split case with $\Lambda(1,\varpi)=-\omega=\Lambda(\varpi,1)$. In this latter case $B(1)=0$, but $B(\hat u_is_1s_2)\neq0$ for $i=1,2$. Here, the elements $\hat u_i$ are defined in (\ref{hatu1u2defeq}).
  \item If $\pi$ is of type IVc, Vb, VIa or VIb, then $B(h(0,m_0))\neq0$.
 \end{enumerate}
\end{theorem}
\begin{proof}
Let $B$ be a non-zero $P_1$-invariant vector in the $(\Lambda,\theta)$-Bessel model of $\pi$. Then $B$ is an element of the space $\mathcal{S}(\Lambda,\theta,P_1)$. Since the space of $P_1$-invariant vectors in $\pi$ is one-dimensional, $B$ is an eigenvector for $T_{1,0}$, $T_{0,1}$ and $\eta$; let $\lambda$, $\mu$ and $\omega$ be the respective eigenvalues. Lemma \ref{onedimlemma} proves our assertions in case $m_0>0$. For the rest of the proof we will therefore assume that $m_0=0$, meaning that $\Lambda$ is unramified. 

i) In this case we have $\lambda = -q\omega$. If we are not in the split case, or if we are in the split case and $\Lambda(1,\varpi) \neq -\omega$ then our assertion follows from Lemma \ref{onedimlemma} i). Assume that we are in the split case and that $\Lambda(1,\varpi)=-\omega=\Lambda(\varpi,1)$. Then $B(1)=0$ by Proposition \ref{IIasplitprop} and (\ref{BBpat1eq}). Hence our assertion follows from Lemma \ref{onedimlemma} i).

ii) If $\pi$ is of type VIb, then $B$ satisfies (\ref{VIbspecialeq}) and the assertion follows from Lemma \ref{onedimlemma} ii). Assume that $\pi$ is of type IVc, Vb or VIa. If we are not in the split case, our assertions follow from Lemma \ref{onedimlemma} i). Assume we are in the split case, and that $\pi$ is of type IVc or Vb. Then, by Table \ref{besselmodelstable} and Table \ref{Iheckeeigenvaluestable}, we have $\Lambda(1,\varpi) \neq -\omega$. Hence our assertions follow from Lemma \ref{onedimlemma} i). Assume we are in the split case, and that $\pi$ is of type VIa. Then our assertion follows from Proposition \ref{VIasplitprop} and (\ref{BBpat1eq}).
\end{proof}
\section{The two-dimensional cases}\label{twodimsec}
In this section, let $\pi$ be an irreducible, admissible representation of $\GSp_4(F)$ of type IIIa or IVb. In both cases the space of $P_1$-invariant vectors is two-dimensional.
%
%
\subsection*{The IIIa case}
Let $B_1$ and $B_2$ be common eigenvectors for $T_{1,0}$ and $T_{0,1}$ in the IIIa case. Then, one can choose the normalizations such that
  \begin{alignat}{2}\label{IIIaB1B2eq}
   T_{1,0}\,B_1&=\alpha\gamma q\,B_1,\qquad&
    T_{1,0}\,B_2&=\gamma q\,B_2,\nonumber\\
   T_{0,1}\,B_1&=\alpha\gamma^2(\alpha q+1)q\,B_1,\qquad&
    T_{0,1}\,B_2&=\alpha\gamma^2(\alpha^{-1}q+1)q\,B_2,\\
   \eta\,B_1&=\alpha\gamma\,B_2,\qquad&
    \eta\,B_2&=\gamma\,B_1.\nonumber
  \end{alignat}
Note that $\alpha \neq 1$ in the IIIa case; see Table \ref{satakenotationtable}.
\begin{lemma}\label{twodimIIIalemma}
 Let $\Lambda$ be a character of $L^\times$, and $m_0$ as in (\ref{m0defeq}). Assume that $B_1,B_2\in\mathcal{S}(\Lambda,\theta,P_1)$ satisfy (\ref{IIIaB1B2eq}) with $\alpha\neq1$. Assume also that $\Lambda(\varpi)=\alpha\gamma^2$. Then
   $$
    B_1=0\qquad\Longleftrightarrow\qquad B_1(h(0,m_0))=0.
   $$
\end{lemma}
\begin{proof}
Using the Atkin-Lehner relations (\ref{IIIaB1B2eq}) and Lemma \ref{T10actionlemma}, we get 
for any $l \geq -1$ and $m \geq 0$,
\begin{equation}\label{B1-s2s1s2-l-l+1}
\gamma B_1(h(l,m)s_2s_1s_2)=q^2B_1(h(l+1,m)s_2s_1s_2),
\end{equation}
and, for $l \geq 0$ and $m \geq {\rm max}(m_0,1)$, the three equations
\begin{align}
 \label{3aT10eq318}0 &= -\alpha\gamma q B_1(h(l,m)s_1s_2) + q^2B_1(h(l-1,m+1)s_1s_2)
  +\alpha \gamma (q-1)B_1(h(l,m)),\\
 \label{3aT10eq328}0 &= -\gamma q B_1(h(l+1,m-1)s_2) + q^2 B_1(h(l,m)s_2)
  +q^2(q-1) B_1(h(l,m)s_2s_1s_2),\\
 \label{3aT10eq428}0 &= q^3 B_1(h(l,m)s_2s_1s_2) + q^2 B_1(h(l,m)s_2) + q B_1(h(l,m)s_1s_2) + B_1(h(l,m)).
\end{align}
  
We will first consider the case ${\bf m_0 > 0}$. Assume that $B_1(h(0,m_0))=0$; we will show that $B_1=B_2=0$. Note that, by the Atkin-Lehner relations, we only need to show that $B_1=0$. By Proposition \ref{maintowerprop}, we know that $B_1(h(l,m)) = 0$ for all $l,m \geq 0$. Considering (\ref{3aT10eq428}) with $l=0$ and $m=m_0$, Lemma \ref{T01actionlemma} ii) with $B=B_2$, $l=0$ and $m=m_0-1$, (\ref{T01s2conslemmaeq1b}) applied to $B_1$, (\ref{T01s2conseq4b}) applied to $B_1$, Lemma \ref{T01actionlemma} ii) with $B=B_2$, $l=0$ and $m=m_0$, Lemma \ref{T01actionlemma} ii) with $B=B_1$, $l=0$ and $m=m_0-1$, (\ref{T01s2conslemmaeq1b}) applied to $B_2$ and Lemma \ref{T01actionlemma} ii) with $B=B_1$, $l=0$ and $m=m_0$, 
we get a homogeneous system of $8$ linear equations in the $7$ variables
\begin{align}\label{lin-sys-var}
 &B_1(h(0,m_0-1)s_2),\;B_1(h(0,m_0)s_2),\;B_1(h(1,m_0)s_2),\;B_1(h(0,m_0)s_1s_2),\nonumber\\ 
 &B_1(h(1,m_0)s_1s_2),\;B_1(h(-1,m_0)s_1s_2),\;B_1(h(0,m_0)s_2s_1s_2).
\end{align}
For any $\alpha$, either the set of first $7$ equations or the set of last $7$ equations has a non-singular matrix. Hence, all the values in (\ref{lin-sys-var}) are zero. Now (\ref{T01s2conslemmaeq1}) and (\ref{3aT10eq328}) imply that
$
 B_1(h(l,m)s_2) = 0$  for all $l \geq 0, \,\,m = m_0-1$ or $m=m_0$.
Using (\ref{B1-s2s1s2-l-l+1}) and (\ref{3aT10eq428}), we get
$B_1(h(l,m_0)s_2s_1s_2) = B_1(h(l,m_0)s_1s_2) = 0$ for $l \geq -1$. Now, using (\ref{3aT10eq318}), (\ref{3aT10eq328}), (\ref{3aT10eq428}), induction,  Proposition \ref{HFRFSipdecomplemma} and the automatic vanishing from Lemma \ref{automaticvanishinglemma}, it follows that $B_1=0$. This concludes our proof in case $m_0>0$.

We next consider the case ${\bf m_0 = 0}$. Assume that $B_1(1)=0$; we will show that $B_1=0$. By Proposition \ref{maintowerprop}, we know that $B_1(h(l,m)) = 0$ for all $l,m \geq 0$. Suppose we can show that $B_1$ vanishes on all the double coset representatives in Proposition \ref{HFRFSipdecomplemma} that have $m=0$ and that $B_1(h(l,1)s_1s_2) = 0$ for all $l\geq-1$. Using induction and (\ref{3aT10eq318}), we get $B_1(h(l,m)s_1s_2) = 0$ for all $l\geq-1$ and $m\geq0$. Now, induction and (\ref{3aT10eq328}), (\ref{3aT10eq428}), gives us $B_1 = 0$, and hence, $B_2 = 0$. 

We will give the proof of the inert case $\big(\frac L{\p}\big) = -1$ here. The other cases are similar (see \cite{Pitale-Schmidt-preprint-2012}). Using Lemma \ref{T10actionlemma} ii), we get, for $l \geq 0$,
\begin{equation}\label{m0=0-inert-3a-1}
 \alpha \gamma q B_1(h(l,0)s_2) = q^2 B_1(h(l-1,1)s_1s_2).
\end{equation}
Using (\ref{B1-s2s1s2-l-l+1}) and Lemma \ref{T10actionlemma} iv), we get, for $l \geq 0$,
\begin{equation}\label{m0=0-inert-3a-2}
 \alpha \gamma q B_1(h(l,0)s_2s_1s_2) = -(q+1)B_1(h(l-1,1)s_1s_2).
\end{equation}
Hence, from (\ref{B1-s2s1s2-l-l+1}), (\ref{m0=0-inert-3a-1}) and (\ref{m0=0-inert-3a-2}), we get, for $l \geq 0$, 
\begin{equation}\label{m0=0-inert-3a-3}
 \gamma B_1(h(l,0)s_2) = q^2 B_1(h(l+1,0)s_2).
\end{equation}
Using (\ref{m0=0-inert-3a-1}), (\ref{m0=0-inert-3a-3}) and Lemma \ref{T01actionlemma} ii), we get
\begin{equation}\label{m0=0-inert-3a-4}
 \alpha \gamma^2 (\alpha q +1) q B_1(s_2) = q^4 B_1(h(0,1)s_1s_2) = \alpha \gamma q^3 B_1(h(1,0)s_2) = \alpha \gamma^2 q B_1(s_2).
\end{equation}
Since $\alpha \neq 0$ , we get $B_1(s_2) = 0$. Now (\ref{m0=0-inert-3a-1}), (\ref{m0=0-inert-3a-2}) and (\ref{m0=0-inert-3a-3}) implies that $B_1$ vanishes on all the double coset representatives in Proposition \ref{HFRFSipdecomplemma} that have $m=0$. Hence $B_1 =0$, as claimed.
\end{proof}
\subsection*{The IVb case}
Let $B_1$ and $B_2$ be common eigenvectors for $T_{1,0}$ and $T_{0,1}$ in the IVb case. We can choose the normalizations such that
  \begin{alignat}{2}\label{IVbB1B2eq}
   T_{1,0}\,B_1&=\gamma\,B_1,\qquad&
    T_{1,0}\,B_2&=\gamma q^2\,B_2,\nonumber\\
   T_{0,1}\,B_1&=\gamma^2(q+1)\,B_1,\qquad&
    T_{0,1}\,B_2&=\gamma^2q(q^3+1)\,B_2,\\
   \eta\,B_1&=\gamma\,B_2,\qquad&
    \eta\,B_2&=\gamma\,B_1.\nonumber
  \end{alignat}
Recall that $\gamma = \sigma(\varpi)$, where $\sigma$ is an unramified character. From Table \ref{besselmodelstable}, a $(\Lambda,\theta)$-Bessel model exists if and only if $\Lambda = \sigma \circ N_{L/F}$. In particular, the number $m_0$ defined in (\ref{m0defeq}) must be zero, i.e., $\Lambda$ must be unramified. The central character condition is equivalent to $\Lambda(\varpi) = \gamma^2$. Moreover, in the ramified case $\big(\frac L\p\big)=0$, evaluating at $\varpi_L$, we get $\Lambda(\varpi_L)=\gamma$, and in the split case $\big(\frac L\p\big)=1$, evaluating at $(\varpi,1)$ and $(1,\varpi)$, we get $\Lambda(\varpi,1)=\Lambda(1,\varpi)=\gamma$.
\begin{lemma}\label{twodimIVblemma}
 Let $\Lambda$ be an unramified character of $L^\times$ satisfying $\Lambda(\varpi) = \gamma^2$. If $\big(\frac L\p\big)=0$, assume that $\Lambda(\varpi_L)=\gamma$, and if $\big(\frac L\p\big)=1$, assume that $\Lambda(\varpi,1)=\Lambda(1,\varpi)=\gamma$. Assume that $B_1,B_2\in\mathcal{S}(\Lambda,\theta,P_1)$ satisfy (\ref{IVbB1B2eq}). Then
   $$
    B_1=0\quad\Longleftrightarrow\quad B_1(1)=0\quad\Longleftrightarrow\quad B_2(1)=0\quad\Longleftrightarrow\quad B_2=0.
   $$
\end{lemma}
\begin{proof} The proof is similar to that of Lemma \ref{twodimIIIalemma}, $m_0=0$ case, and is therefore omitted. See \cite{Pitale-Schmidt-preprint-2012} for details.
\end{proof}
\subsection*{The main result}
The following theorem identifies good test vectors for those irreducible, admissible representations of $\GSp_4(F)$ that are not spherical but have a two-dimensional space of $P_1$-invariant vectors. Recall from Table \ref{Iwahoritable} that these are precisely the Iwahori-spherical representations of type IIIa and IVb.

\begin{theorem}\label{twodimtheorem}
 Let $\pi$ be an irreducible, admissible, Iwahori-spherical representation of $\GSp_4(F)$ that is not spherical but has a two-dimensional space of $P_1$-invariant vectors. Let $S$ be the matrix defined in (\ref{Sxidefeq}), with $\mathbf{a},\mathbf{b},\mathbf{c}$ subject to the conditions (\ref{standardassumptions}). Let $T(F)$ be the group defined in (\ref{TFdefeq}). Let $\theta$ be the character of $U(F)$ defined in (\ref{thetadefeq}), and let $\Lambda$ be a character of $T(F)\cong L^\times$. Let $m_0$ be as in (\ref{m0defeq}). We assume that $\pi$ admits a $(\Lambda,\theta)$-Bessel model. Then the space of $P_1$-invariant vectors is spanned by common eigenvectors for the Hecke operators $T_{1,0}$ and $T_{0,1}$, and if $B$ is any such eigenvector, then $B(h(0,m_0))\neq0$.
\end{theorem}
\begin{proof}
Assume first that $\pi$ is of type IIIa. Then $\pi$ has a Bessel model with respect to any $\Lambda$; see Table \ref{besselmodelstable}. Let $B_1$ and $B_2$ be the $P_1$-invariant vectors which are common eigenvectors for $T_{1,0}$ and $T_{0,1}$, as in (\ref{IIIaB1B2eq}). Then $B_1(h(0,m_0))\neq0$ by Lemma \ref{twodimIIIalemma}. If we replace $\gamma$ by $\alpha^{-1}\gamma$ and then $\alpha$ by $\alpha^{-1}$ in the equations (\ref{IIIaB1B2eq}), then the roles of $B_1$ and $B_2$ get reversed. This symmetry shows that also $B_2(h(0,m_0))\neq0$.

Now assume that $\pi$ is the representation $L(\nu^2,\nu^{-1}\sigma\St_{\GSp(2)})$ of type IVb. Then, by Table \ref{besselmodelstable}, we must have $\Lambda=\sigma\circ N_{L/F}$. In particular, $\Lambda$ is unramified and satisfies the hypotheses of Lemma \ref{twodimIVblemma}. Let $B_1$ and $B_2$ be the $P_1$-invariant vectors which are common eigenvectors for $T_{1,0}$ and $T_{0,1}$, as in (\ref{IVbB1B2eq}). Then $B_1(1)\neq0$ and $B_2(1)\neq0$ by Lemma \ref{twodimIVblemma}.
\end{proof}

\bibliography{BSI}{}
\bibliographystyle{plain}

\end{document}